\documentclass{article}
\usepackage{amsmath, amssymb, amsthm, biblatex, changepage, faktor, float, graphicx, tabto, caption, subcaption, thmtools, thm-restate, hyperref, cleveref}
\title{\textbf{On Clique Graphs and\\Clique Regular Graphs}}
\author{Robert R. Petro and Connor M. Phillips}
\date{February 2025}

\newtheorem{thm}{Theorem}[section]
\newtheorem{lem}[thm]{Lemma}
\newtheorem{cor}[thm]{Corollary}
\theoremstyle{definition}
\newtheorem{remark}[thm]{Remark}
\newtheorem{exmp}{Example}

\begin{document}

\maketitle

\begin{abstract}
If $\Gamma$ is a graph for which every edge is in exactly one clique of order $\omega$, then one can form a new graph with vertex set equal to these cliques.  This is a generalization of the line graph of $\Gamma$.  We discover many general results and classifications related to these clique graph that will be useful to researchers studying these objects. In particular, we find bounds on its eigenvalues (with exact results when $\Gamma$ is $k$-regular) and some complete classifications when $\Gamma$ is strongly regular. We apply our results to many examples, including Conway’s 99-graph problem and the existence problem for other strongly regular graphs.
\end{abstract}

\medskip

\section{Introduction}
\tab A popular notion in the study of graph theory is the line graph. The \textbf{line graph} of a graph $\Gamma$, also referred to as an interchange graph or an edge graph, shows the adjacency between the edges of $\Gamma$. The line graph is denoted as $L(\Gamma)$. In all but one case, the structure of $\Gamma$ can be completely recovered from its line graph (the $K_3$ graph and the $K_{1,3}$ have the same line graph but are not isomorphic thus the original structure is not fully recoverable) \cite{williams}. The line graph of $\Gamma$ is defined by taking the set of edges of $\Gamma$ as the vertices of $L(\Gamma)$ and two vertices are adjacent if and only if their corresponding edges in $\Gamma$ share a vertex.\par
If $\Gamma$ is $k$-regular with $n$ vertices, then $m = \frac{nk}{2}$ is the number of edges. The eigenvalues of the line graph can be derived from the eigenvalues of the original graph $\{\lambda_i^{a_i}\}$ (where the exponent denotes the multiplicity of the eigenvalue as is customary) as it is well known that the characteristic polynomial of the line graph is given by
\[p(L(G); \lambda)=(\lambda +2)^{m-n}p(\Gamma;\lambda + 2 - k).\]
Therefore the spectrum of $L(\Gamma)$ is
\[-2^{m-n}, (\lambda_i + k - 2)^{a_i}.\]
\tab As $\Gamma$ becomes larger (in both number of vertices and edges) and more complex (add in regularity and number of common neighbor requirements), the line graph of $\Gamma$ becomes exceedingly more intensive to construct. For example, a smaller graph with 10 vertices that is 6-regular has a line graph with 30 vertices and 150 edges, while a larger graph with 99 vertices that is 14-regular has a line graph with 693 vertices and 9009 edges. This issue makes the construction of line graphs from graphs of considerable size difficult. Moreover, if the graph is unknown and will be of considerable size then the construction of a potential line graph in order to obtain $\Gamma$ is infeasible.\\
\tab\tab In this paper, we introduce a generalization of the line graph that we call the clique graph. We use this generalization to avoid the issues that line graphs typically experience with larger and more complex graphs. Using the clique graph, and its properties, we introduce a new classification of graphs that we designate as clique regular graphs.\\
\tab\tab The clique graph, clique regular graphs, and their properties prove to be useful tools in deducing not only information about graphs they are derived from, but also information about the properties of many families of graphs. We demonstrate in Example \ref{exmp:1,2s} that the clique graph of a graph can be used to gather novel information about strongly regular graphs whose existence is unknown.\\
\tab\tab The paper is organized as follows. In the next section we introduce basic principles of graphs, strongly regular graphs, and line graphs we will use. In the subsequent sections we define the idea and construction of the clique graph, and we establish the properties of $\omega$-clique regular graphs. We finish by applying the obtained theorems to families of graphs that have the clique regular property including locally linear graphs.

\medskip

\section{Preliminaries}

\subsection{Basic Graph Theory}
\tab We now define some notions important to our results. For other definitions used, see any standard graph theory reference \cite{harary}. A graph $\Gamma = \Gamma(V,E)$ is a finite set of vertices $V$ coupled with a finite set of pairs of vertices called edges $E$. Two vertices $x$ and $y$ are \textbf{adjacent} if and only if there exists an edge $\{x,y\}$, also denoted $xy$, in $E$. The \textbf{degree} of a vertex $x$ is the number of vertices it is adjacent to and will be denoted $d(x)$. The largest degree of any vertex in a graph $\Gamma$ will be denoted $\Delta(\Gamma)$. A graph $\Gamma$ is \textbf{regular} with degree $k$ ($k$-regular) if the degree of each vertex in $\Gamma$ is $k$. The \textbf{neighborhood} of a vertex $x$ is the set of vertices that are adjacent to vertex $x$ and will be denoted $N(x)$. A \textbf{clique} is a subset of vertices of a graph such that every vertex in the set is adjacent to every other vertex in the set. A clique of order $\omega$ ($\omega$-clique) is a clique containing $\omega$ vertices. The \textbf{clique number} of a graph $\Gamma$ is the maximum order of a clique in $\Gamma$ and it is denoted $\omega(\Gamma)$. A \textbf{complete graph}, $K_n$, is the graph with $n$ vertices such that every vertex is adjacent to every other vertex. Similarly a \textbf{complete bipartite graph}, $K_{n,m}$, is the graph on $n+m$ vertices which are partitioned into two sets of order $n$ and $m$ such that no two vertices in the same set are adjacent and any two vertices not in the same set are adjacent. A graph is called \textbf{locally linear} if it has a nonempty edge set and each edge belongs to a unique triangle, i.e. for every pair of adjacent vertices there is exactly one other common neighbor \cite{harary}. Given a fixed ordering of vertices $V=\{v_1, v_2, \dots, v_n\}$ the \textbf{adjacency matrix} of $\Gamma$ is defined as the $n \times n$ matrix $A$ with rows and columns indexed by the vertices of $\Gamma$, with
\[(A)_{ij} = 
    \begin{cases}
        1 & \text{if $v_i$ is adjacent to $v_j$}\\
        0 & \text{otherwise}
    \end{cases}\]
    where $(A)_{ij}$ is the $i, j^\text{th}$ entry of $A$. We denote by $I_n$ the $n\times n$ identity matrix. The \textbf{spectrum} of a graph $\Gamma$ is the set of eigenvalues of the adjacency matrix of $\Gamma$ along with their multiplicities. It is denoted $\lambda_1^{a_1}, \ldots , \lambda_r^{a_r}$ where the exponent represents the multiplicity of that eigenvalue.

\subsection{Strongly Regular Graphs}
\tab A graph $\Gamma$ is called \textbf{strongly regular} with parameters ($n,k,\lambda,\mu$) and is abbreviated as srg($n,k,\lambda,\mu$) if
\begin{itemize}
    \item $\Gamma$ has $n$ vertices,
    \item $\Gamma$ is $k$-regular,
    \item Adjacent vertices share exactly $\lambda$ common neighbors, and
    \item Distinct non-adjacent vertices share exactly $\mu$ common neighbors.
\end{itemize}
\tab The parameters ($n,k,\lambda,\mu$) are closely related and must obey, 
\[(n-k-1)\lambda=k(k-\mu-1).\]
\tab A level below strongly regular graphs are edge regular graphs which are denoted as $erg(n,k,\lambda)$. A graph is \textbf{edge regular} with parameters $(n,k,\lambda)$ if the first three conditions of a strongly regular graph hold. All strongly regular graphs are edge regular but not all edge regular graphs are strongly regular.\\
\tab\tab A graph $\Gamma$ is strongly regular if and only if it is regular with 3 or fewer distinct eigenvalues \cite{haemers}. So we can classify strongly regular graphs as either ``boring'' (graphs with one or two eigenvalues), disjoint unions of complete graphs and their complements, or ``non-boring'' if otherwise. Any non-boring srg$(n,k,\lambda, \mu)$ will have spectrum $k^1, r^f, s^g$ where $k>r>s$ and the latter two eigenvalues and their multiplicities are given by the following,
\[r = \frac{1}{2}\left[(\lambda - \mu) + \sqrt{(\lambda - \mu)^2+4(k-\mu)}\right],\]
\[f = \frac{1}{2}\left[(n-1)-\frac{2k + (n-1)(\lambda-\mu)}{\sqrt{(\lambda - \mu)^2+4(k-\mu)}}\right],\]
\[s = \frac{1}{2}\left[(\lambda - \mu) - \sqrt{(\lambda - \mu)^2+4(k-\mu)}\right],\]
\[g = \frac{1}{2}\left[(n-1)+\frac{2k + (n-1)(\lambda-\mu)}{\sqrt{(\lambda - \mu)^2+4(k-\mu)}}\right].\]\\
From these it is also provable that $\lambda -\mu = r+s$ and $k-\mu = -rs$.\\
\tab\tab One of the biggest questions on the topic of strongly regular graphs is for which sets of parameters ($n,k,\lambda,\mu$) does a strongly regular graph exist. A large list of feasible parameter sets can be found on Andries Brouwer's website \cite{brouwer}. Brouwer's list categorizes each set of srg parameters into three groups (existence, non-existence, and unknown). Boring graphs are excluded from this list.

\medskip

\section{Clique Graphs}

\tab The \textbf{$\omega$-clique graph} of a graph shows the adjacency between the cliques of order $\omega$ of the graph. This graph will be denoted $C_\omega(\Gamma)$ and is defined in the following manner:
\begin{enumerate}
    \item The vertices of the clique graph $C_\omega(\Gamma)$ are the cliques of order $\omega$ in $\Gamma$.
    \item Two distinct cliques are adjacent in $C_\omega(\Gamma)$ if and only if they have a nonempty intersection.
\end{enumerate}
\begin{figure}[h!]
\begin{subfigure}{.5\textwidth}
  \centering
  \includegraphics[width=.6\linewidth]{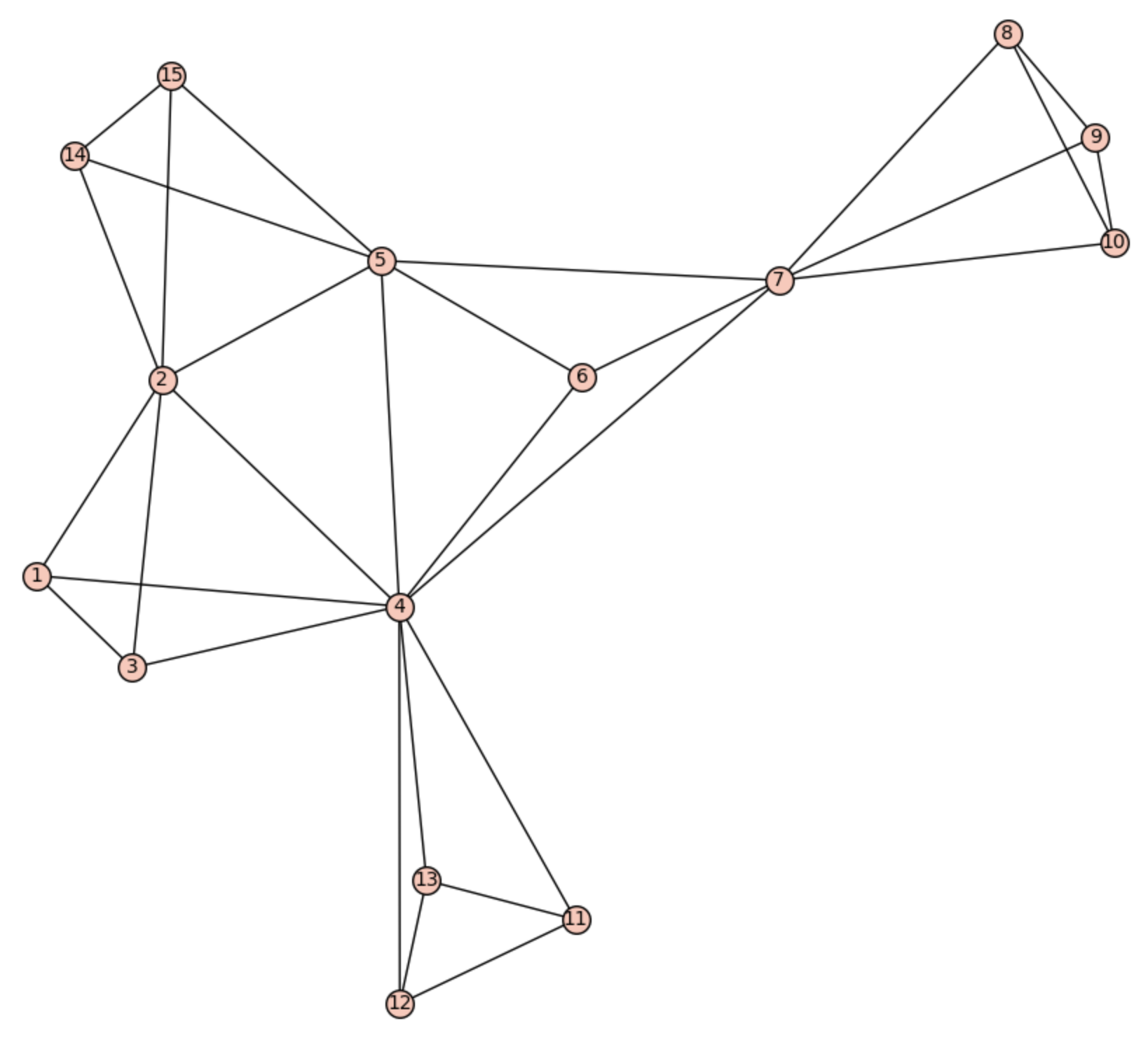}
  \caption{A graph, $\Gamma$}
\end{subfigure}
\begin{subfigure}{.5\textwidth}
  \centering
  \includegraphics[width=.6\linewidth]{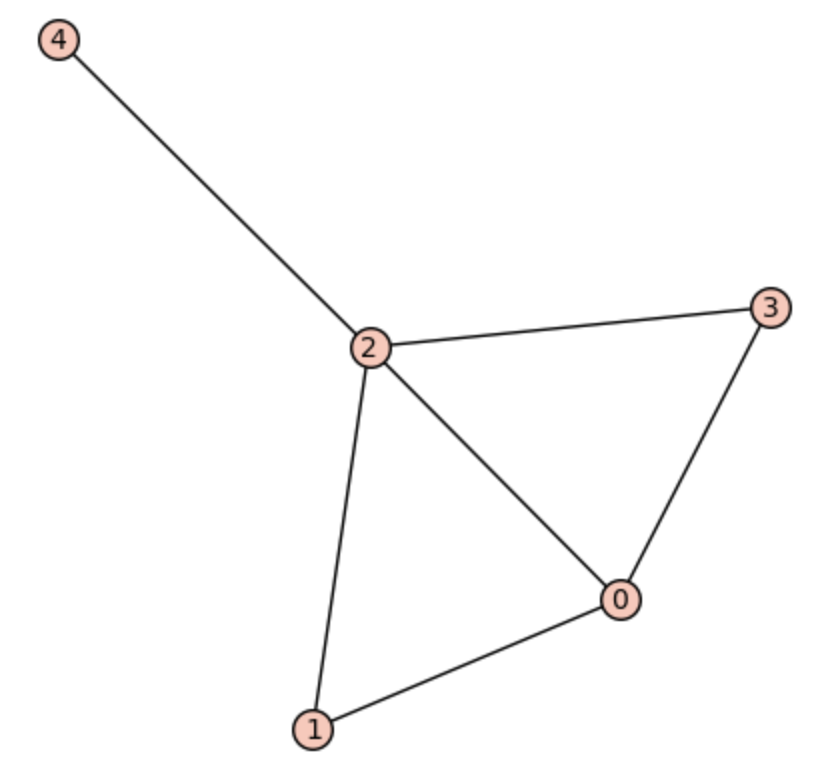}
  \caption{The $4$-clique graph of $\Gamma$, $C_4(\Gamma)$}
\end{subfigure}
\caption{}
\end{figure}
\tab Since a $2$-clique is just a single edge, it is clear that for any graph $\Gamma$, its line graph is isomorphic to its $2$-clique graph, $C_2(\Gamma) \cong L(\Gamma)$. For this reason, we consider the clique graph construction to be a generalization of the line graph and taking $\omega =2$ for any of our theorems reveals a fact about the line graph.\\

\subsection{Clique Regular Graphs}

\tab An \textbf{$\omega$-clique regular} graph is a graph with a nonempty edge set such that every edge is in a unique clique of order $\omega$. Again, since a 2-clique is just a single edge every graph with a nonempty edge set is 2-clique regular. Notice also that since a $3$-clique is a triangle, the locally linear graphs are exactly the 3-clique regular graphs. For this reason we consider the clique regular property to be a generalization of the locally linear property.
\begin{thm}
    If $\Gamma$ is an $\omega$-clique regular graph, then $\omega=2$ or $\omega=\omega(\Gamma)$.
\end{thm}
\begin{proof}
    Assume $\Gamma$ is $\omega$-clique regular and $\omega \neq 2$. If $2 < \omega < \omega(\Gamma)$, then there exists a clique in $\Gamma$ of order $\omega +1$ and every edge in this clique is in at least $\omega -1$ cliques of order $\omega$ in $\Gamma$. If $\omega(\Gamma) < \omega$, then there are no $\omega$-cliques in $\Gamma$. So $\omega$ must equal $\omega(\Gamma)$.
\end{proof}

Notice that the definition of clique regularity is similar to the definition of a \textbf{regular clique assembly} as given by Guest et al \cite{guest}. They define a regular clique assembly as a regular graph with clique number greater than or equal to 2 such that every maximal clique is maximum, and each edge is in exactly one maximum clique. Clearly every regular clique assembly $\Gamma$ is $\omega(\Gamma)$-clique regular and the following theorem will classify when clique regular graphs are regular clique assemblies. 

\begin{thm} \label{thm:rca}
    Suppose $\Gamma$ is an $\omega$-clique regular graph on $n$ vertices. Then $\Gamma$ is a regular clique assembly if and only if $\Gamma$ is an erg$(n,\Delta(\Gamma),\omega -2)$.
\end{thm}
\begin{proof}
    The proof that any regular clique assembly is an erg$(n,\Delta(\Gamma),\omega -2)$ comes from Guest et al \cite[pg. 304]{guest}. So now assume that $\Gamma$ is $\omega$-clique regular and an erg$(n,k,\omega-2)$ and we will show that $\Gamma$ is a regular clique assembly. If $\omega =2$ or $\omega=3$ this proof also comes from Guest et al \cite[pg. 304]{guest} so assume that $\omega$ is the clique number of $\Gamma$ and is greater than 3. Since $\Gamma$ is regular from the hypothesis, it is sufficient to show that every maximal clique in $\Gamma$ has order $\omega$. Suppose for contradiction there exists a maximal clique of order less than $\omega$, $c_<$, and let $x$ and  $y$ be vertices in that clique. Then the edge $xy$ is in a unique clique of order $\omega$, $c_\omega$, so $x$ and $y$ have $\omega -2$ neighbors in that clique. Since $c_<$ is not contained in $c_\omega$, there exists a vertex $z$ in $c_<$ and not in $c_\omega$. So $z$ is also a common neighbor of $x$ and $y$ which implies the order of $N(x)\cap N(y)$ is greater than $\omega -2$, a contradiction.
\end{proof}
This theorem will be useful in establishing the connection between clique regular graphs and strongly regular graphs.
Many well known families of graphs are clique regular. We will formally prove these in Section 4 but they include the following:
\begin{itemize}
    \item Square Rook Graphs -  Each vertex of the graph represents a point on a square grid, and there is an edge between any two points that share a row or column.
    \begin{figure}[H]
        \centering
        \includegraphics[width=0.4\linewidth]{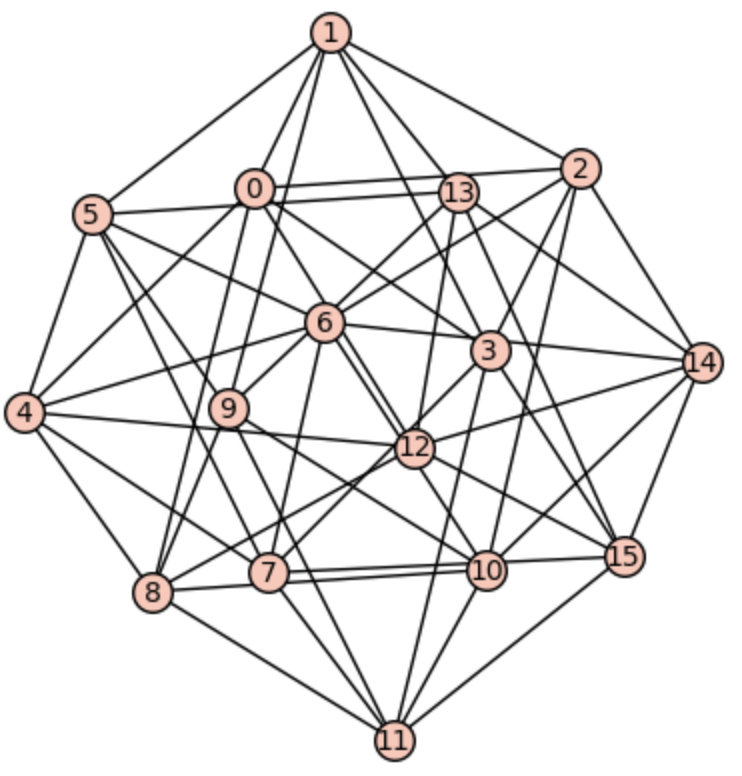}
        \caption{$5\times5$ rook graph}
        \label{fig:5x5_rook}
    \end{figure}

    \item Orthogonal Array Block Graphs - The orthogonal array block graph is the graph with vertices as the $m \times 1$ column vectors of $OA(n, m)$, where two vectors are adjacent if and only if they have nonempty intersection.
    \begin{figure}[H]
        \centering
        \includegraphics[width=0.5\linewidth]{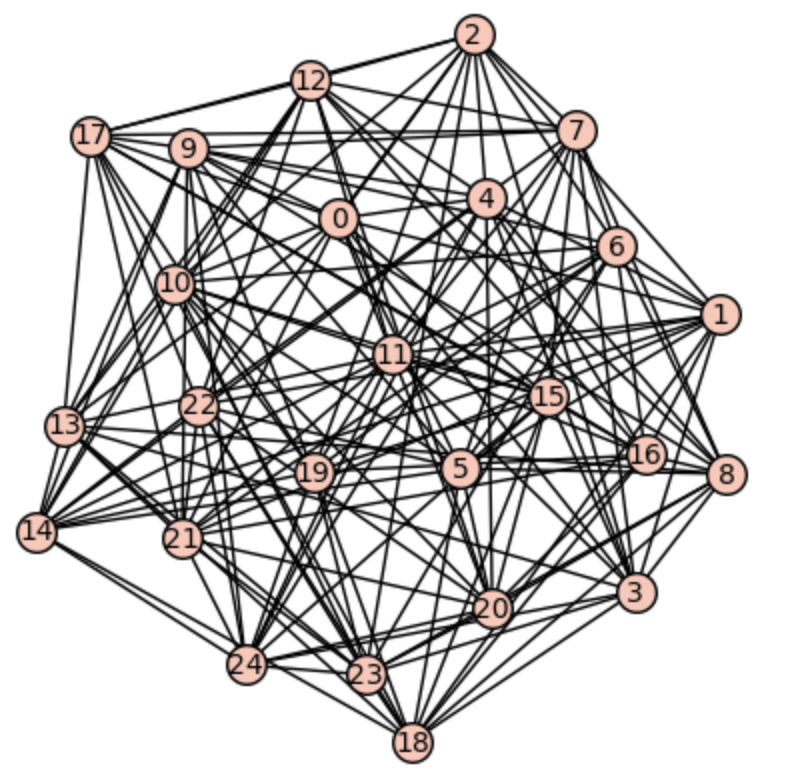}
        \caption{The blokc graph of a OA(5,3)}
    \end{figure}
    \item Generalized Quadrangle Collinearity Graphs - A generalized quadrangle GQ($s,t$) is a point-line incidence structure that satisfies the following properties for some $s,t\geq 1$: every line has $s + 1$ points, every point lies on $t + 1$ lines, there is at most one point on any two distinct lines, and if $P$ is a point not on line $l$, then there is a unique line incident with $P$ and meeting $l$.
    \begin{figure}[H]
        \centering
        \includegraphics[width=0.5\linewidth]{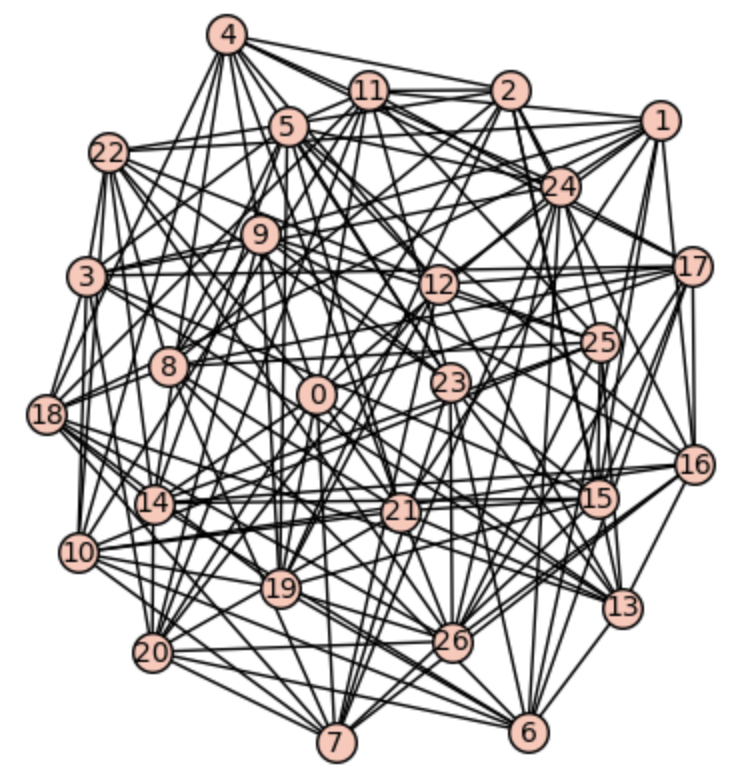}
        \caption{The Collinearity Graph of a GQ(2,4)}
    \end{figure}
    \item Triangular Graphs - These graphs, denoted $T_n$, are the line graphs of the complete graph $K_n$.
    \begin{figure}[H]
        \centering
        \includegraphics[width=0.5\linewidth]{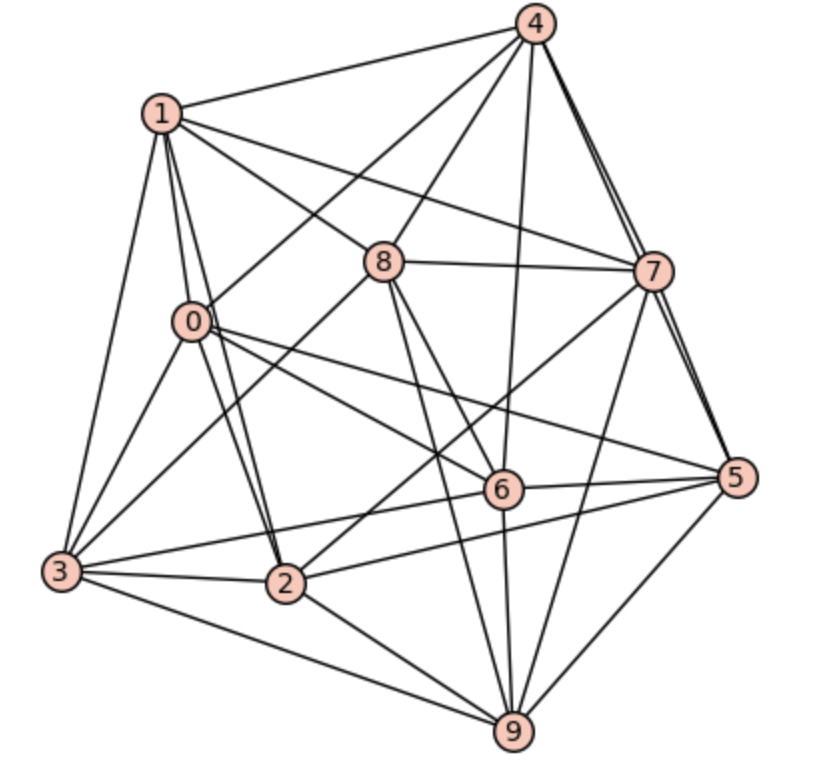}
        \caption{$T_5$ Graph}
    \end{figure}
\end{itemize}

The square rook graphs and triangular graphs are each an example of a line graph $L(\Gamma)$ that is $\omega$-clique regular and has an $\omega$-clique graph that is isomorphic to the original graph $\Gamma$. The following lemmas and theorems will classify for $w \geq 3$, all graphs which have the first property and all connected graphs that have the second property.
\begin{lem} \label{lem:d(v)}
    For each vertex $v$ in $\Gamma$, the set of edges incident to $v$ form a clique in $L(\Gamma)$ with order $d(v)$. If $d(v) > 2$, then this clique is maximal.
\end{lem}
\begin{proof}
    Vertex $v$ is incident to $d(v)$ edges in $\Gamma$ all of which share one endpoint $v$. So in $L(\Gamma)$ those edges incident to $v$ form a clique of order $d(v)$. Now assume this clique is not maximal and we will show $d(v) \leq 2$. Let $e$ be an edge in $\Gamma$ that is not incident to $v$ but is adjacent to all other edges incident to $v$. Then $e$ must be incident to the other end points of these edges, so $v$ can be incident to at most 2 edges.
\end{proof}
Following from Lemma \ref{lem:d(v)}, for any vertex $v \in \Gamma$ when we refer to the ``clique created by $v$" we mean the clique in $L(\Gamma)$ of order $d(v)$ induced by the edges in $\Gamma$ incident to $v$.
\begin{lem} \label{lem:whit}
    For $\omega>3$, $L(\Gamma) \cong K_\omega$ if and only if $\Gamma \cong K_{1,\omega}$. Also, $L(\Gamma) \cong K_3$ if and only if $\Gamma \cong K_3$ or $\Gamma \cong K_{1,3}$.
\end{lem}
\begin{proof}
    The graph $K_{1,\omega}$ has $\omega$ edges that all share a common end point, so $L(K_{1,\omega}) \cong K_\omega$ for $\omega \geq 3$. Also, it is clear that $L(K_3) \cong K_3$. The converse for both cases follows from the Whitney Isomorphism Theorem \cite{whitney}.
\end{proof}
\begin{remark} \label{remark}
    Lemma \ref{lem:whit} implies that the only structures that can create a maximal clique in $L(\Gamma)$ are the set of edges incident to a vertex $v$, i.e. cliques created by a vertex $v$ following from Lemma \ref{lem:d(v)}. Except in the case of a 3-clique in $L(\Gamma)$ which can also be created by a triangle, $K_3$.
\end{remark}
We will begin by classifying graphs that have $\omega$-clique regular line graphs and start with the case $\omega \geq 4$.
\begin{thm} \label{thm:4cr}
    Suppose $\omega \geq 4$ and $\Gamma$ is a graph. Then $L(\Gamma)$ is $\omega$-clique regular if and only if the degree of every vertex in $\Gamma$ is 1 or $\omega$.
\end{thm}
\begin{proof}
    First assume that the degree of every vertex in $\Gamma$ is 1 or $\omega$. Let $e=\{xv, yv\}$ be an edge in $L(\Gamma)$. Then the vertex $v$ in $\Gamma$ has degree $\omega$ and so by Lemma \ref{lem:d(v)} $e$ is in the maximal $\omega$-clique created by $v$, and by Remark \ref{remark}, this is the unique $\omega$-clique $e$ is in.\\
    For the inverse, assume there exists a vertex $v$ in $\Gamma$ with degree not equal to either 1 or $\omega$. If $1 < d(v) < \omega$, then there exists edges $e_1$ and $e_2$ incident to $v$ in $\Gamma$. So the edge $e_1e_2$ in $L(\Gamma)$ is in the maximal clique created by $v$ of order less than $\omega$ and $e_1$ and $e_2$ share no other endpoints in $\Gamma$ so by Remark \ref{remark}, the edge $e_1e_2$ is in no clique of order $\omega$. If $d(v) > \omega $, then there exists a clique or order greater than $\omega$ in $L(\Gamma)$. So in either case $L(\Gamma)$ is not $\omega$-clique regular.
\end{proof}
The case of the $3$-clique regular line graph is more complicated because\linebreak $L(K_3) \cong K_3 \cong L(K_{1,3})$.
\begin{thm} \label{thm:3cr}
    Suppose $\Gamma$ is a graph with connected components $C_1, C_2, \ldots , C_n$. Then $L(\Gamma)$ is $3$-clique regular if and only if whenever $C_i \not \cong K_3$ , $C_i$ is triangle free and the degree of every vertex in $C_i$ is 1 or $3$. 
\end{thm}
\begin{proof}
    First, assume that for each $C_i \not \cong K_3$, $C_i$ is triangle free and the degree of every vertex in $C_i$ is 1 or $3$. Let $e=\{xv,yv\}$ be an edge in $L(\Gamma)$ with vertex $v$ in connected component $C_k$. If $C_k \cong K_3$, then edge $e$ is in $L(C_k) \cong K_3$ a connected component of $L(\Gamma)$ so we are done. So assume $C_k \not \cong K_3$ implying by our assumption that the degree of $v$ is $3$. So the $3$-clique created by $v$ is maximal by Lemma \ref{lem:d(v)} and since $C_i$ is triangle free by our assumption, Remark \ref{remark} implies that this is the only $3$-clique containing edge $e$.\\
    Now assume the inverse, for some $C_i \not \cong K_3$, $C_i$ contains a triangle or $C_i$ contains a vertex with degree not equal to either $1$ or $3$. For the first case, since $C_i$ is connected, contains a triangle and is not $K_3$, Figure \ref{fig:sub1} is a sub graph of $\Gamma$ and therefore Figure \ref{fig:sub2} is a sub graph of $L(\Gamma)$. So $L(\Gamma)$ is not $3$-clique regular.
    \begin{figure}[H]
    \begin{subfigure}{.5\textwidth}
      \centering
      \includegraphics[width=.8\linewidth]{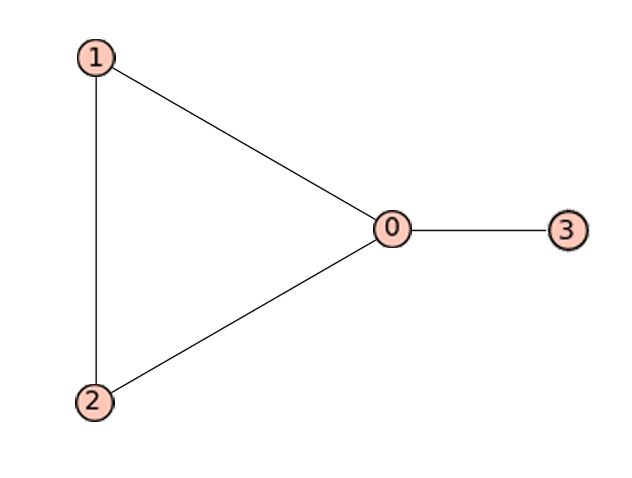}
      \caption{Subgraph of $\Gamma$}
      \label{fig:sub1}
    \end{subfigure}
    \begin{subfigure}{.5\textwidth}
      \centering
      \includegraphics[width=.8\linewidth]{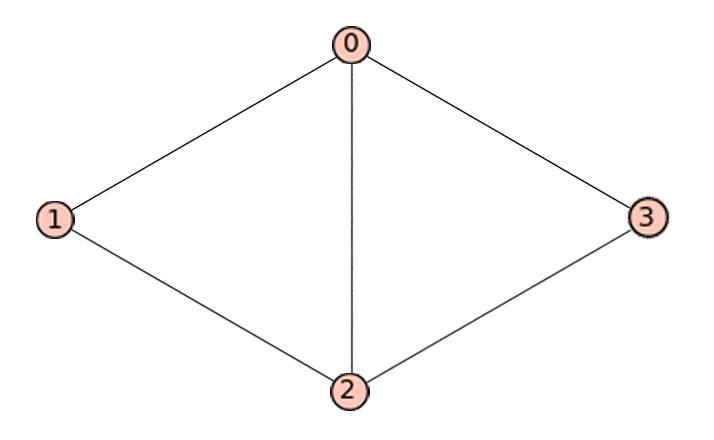}
      \caption{Subgraph of $L(\Gamma)$}
      \label{fig:sub2}
    \end{subfigure}
    \caption{}
    \end{figure}
So now for the second case, let $v$ be a vertex in $C_i$ with degree not equal to either $1$ or $3$. If $d(v) > 3$, then there exists a clique of order greater than $3$ in $L(\Gamma)$ so $L(\Gamma)$ is not $3$-clique regular. If $d(v)=2$, then let vertices $x$ and $y$ be adjacent to $v$. If $x$ and $y$ are adjacent, then $\{v,x,y\}$ is a triangle in $\Gamma$, a contradiction. The edge $e = \{xv,yv\}$ in $L(\Gamma)$ is then not in a $3$-clique induced by the edges incident to $v$ and not in a $3$-clique induced by a triangle in $\Gamma$. So by Remark \ref{remark}, edge $e$ is not in any $3$-clique thus $L(\Gamma)$ is not $3$-clique regular.
\end{proof}
Now we restrict our focus to connected graphs and will classify all that are isomorphic to the $\omega$-clique graph of their line graph. Again, we start with the case $\omega \geq 4$.
\begin{thm} \label{thm:4ci}
    Suppose $\omega \geq 4$ and $\Gamma$ is a connected graph. Then $C_\omega(L(\Gamma)) \cong \Gamma$ if and only if $\Gamma$ is $\omega$-regular.
\end{thm}

\begin{proof}
    First assume $\Gamma$ is $\omega$-regular.\par
    From Lemma \ref{lem:d(v)} there is an injective mapping from the vertices in $\Gamma$ to the $\omega$-cliques their edges induce in $L(\Gamma)$. From Remark \ref{remark} we know that the cliques created by each vertex in $\Gamma$ are the only $\omega$-cliques in $\Gamma$. So the injective mapping is a bijection from the vertices in $\Gamma$ to the vertices in $C_\omega(L(\Gamma))$. We will show this bijection preserves adjacency. If vertices $v$ and $u$ are adjacent in $\Gamma$, then the cliques created by $v$ and by $u$ share the vertex $vu$ in $L(\Gamma)$. So in $C_\omega(L(\Gamma))$ these two cliques are adjacent. If $v$ and $u$ are not adjacent, then their cliques in $L(\Gamma)$ don't share a vertex so they are not adjacent in $C_\omega(L(\Gamma))$. Thus it follows $C_\omega(L(\Gamma)) \cong \Gamma$.\par
    Now assume that $\Gamma \cong C_\omega(L(\Gamma))$.\\
    We will begin by showing $\Delta(\Gamma) \leq \omega$ if $\omega \geq 3$, (we will reference this argument again in Theorem \ref{thm:3ci}).
    First, assume for contradiction $\Delta(\Gamma) \geq \omega +2$ and let $v$ be a vertex in $\Gamma$ with degree $\Delta=\Delta(\Gamma)$. Observe that the number of $\omega$-cliques in the $\Delta$-clique created by $v$ in $L(\Gamma)$ is $\binom{\Delta}{\omega}$. Let $c_k$ be one of these $\omega$-cliques and we will show that the degree of $c_k$ in $C_\omega(L(\Gamma))$ is at least $\binom{\Delta}{\omega}-\binom{\Delta-\omega}{\omega}-1$. In the $\Delta$-clique created by $v$, $c_k$ is only non-adjacent with another clique if they don't share any vertices. There are $\Delta-\omega$ vertices in the $\Delta$-clique created by $v$ that are not in $c_k$ so there are $\binom{\Delta-\omega}{\omega}$ other $\omega$-cliques not adjacent to $c_k$. Implying that $c_k$ is adjacent to $\binom{\Delta}{\omega}-\binom{\Delta-\omega}{\omega}-1$ other $\omega$-cliques in the $\Delta$-clique created by $v$. So $d(c_k) \geq \binom{\Delta}{\omega}-\binom{\Delta-\omega}{\omega}-1$. By induction on $\Delta \geq \omega+2$, we will show that $\binom{\Delta}{\omega}-\binom{\Delta-\omega}{\omega}-1 > \Delta$ which implies $d(c_k) > \Delta$ contradicting $C_\omega(L(\Gamma)) \cong \Gamma$. For the base case let $\Delta=\omega+2$. Since $\omega \geq 3$ then $\binom{\Delta-\omega}{\omega}=\binom{2}{\omega}=0$. So
    \begin{align*}
        \frac{\omega+1}{2} &\geq 2\\
        \frac{(\omega+2)(\omega+1)}{2} -1 & \geq 2\omega+3 >\omega+2\\
        \binom{\omega+2}{\omega}-\binom{2}{\omega}-1 &>\omega+2.
    \end{align*}
    For the inductive step, assume $\binom{\Delta}{\omega}-\binom{\Delta-\omega}{\omega}-1>\Delta$ for some $\Delta\geq \omega+2$ and we will show $\binom{\Delta+1}{\omega}-\binom{\Delta+1-\omega}{\omega}-1>\Delta+1$. Recall that for any positive integers $a$ and $b$, $\binom{a+1}{b}=\binom{a}{b}+\binom{a}{b-1}$. Then since $\binom{\Delta}{\omega-1} > \binom{\Delta-\omega}{\omega-1}$ implies $\binom{\Delta}{\omega-1} - \binom{\Delta-\omega}{\omega-1} >0$ we get,
    \[\binom{\Delta+1}{\omega}-\binom{\Delta+1-\omega}{\omega}-1 > \binom{\Delta}{\omega}-\binom{\Delta-\omega}{\omega}-1 \geq \Delta+1.\]
    Now for contradiction, assume $\Delta(\Gamma) = \omega +1$ and let $v$ be a vertex in $\Gamma$ with degree $\omega +1$. Then $v$ creates a $(\omega +1)$-clique in $L(\Gamma)$ and observe that every $\omega$-clique in the $(\omega+1)$-clique created by $v$ is adjacent to every other $\omega$-clique. So there is a $(\omega+1)$-clique in $C_\omega(L(\Gamma))$ since the number of $\omega$-cliques in a $(\omega+1)$-clique is $\omega+1$. Because of the assumed isomorphism, there must be a $(\omega+1)$-clique in $\Gamma$. Then since $\Delta(\Gamma)=\omega+1$ and $\Gamma$ is connected, then at least one of the vertices in this $(\omega+1)$-clique has degree $\omega+1$, call this vertex $u$. Consider the set of edges in the $(\omega+1)$-clique incident with $u$. This set induces an $\omega$-clique in $L(\Gamma)$ we will call $c_k$ which we will show has degree in $C_\omega(L(\Gamma))$ greater than $\omega +1$. Since the other endpoint of every edge in $c_k$ has degree at least $\omega$, they all create at least one $\omega$-clique in $L(\Gamma)$ which is adjacent to $c_k$. Since $u$ has degree $\omega+1$ in $\Gamma$, it creates $\omega$ other $\omega$-cliques in $L(\Gamma)$ also adjacent to $c_k$. So $d(c_k) \geq 2\omega > \omega +1 = \Delta(\Gamma)$ since $\omega \geq 3$, contradicting the isomorphism.\\
    Let $d_i$ denote the number of vertices in $\Gamma$ with degree $i$ for $1 \leq i \leq \omega$. By Remark \ref{remark} and since $\Delta(\Gamma) \leq \omega$, we know the number of vertices in $C_\omega(L(\Gamma))$ equals $d_\omega$. We also know the number of vertices in $\Gamma$ is $\sum_{i=1}^\omega d_i$. Because these graphs are isomorphic, we have $d_\omega = \sum_{i=1}^\omega d_i$ implying $d_i = 0$ for all $i\neq \omega$. Thus $\Gamma$ is $\omega$-regular.
\end{proof}
Once again, the case of the $3$-clique graph of a line graph is much more complicated even when restricting to connected graphs.
\begin{thm} \label{thm:3ci}
    Suppose $\Gamma$ is a connected graph. Then $C_3(L(\Gamma)) \cong \Gamma$ if and only if:\\
           \hspace*{0.5cm} (1) the degree of every vertex in $\Gamma$ is $2$ or $3$,  \\
           \hspace*{0.5cm} (2) every degree $2$ vertex in $\Gamma$ is contained in a triangle, \\
           \hspace*{0.5cm} (3) every triangle in $\Gamma$ contains exactly one degree $2$ vertex, and \\
           \hspace*{0.5cm} (4) distinct triangles in $\Gamma$ share no vertices.\vspace{0.2cm} 
\end{thm}
The simplest example of a graph with this property is shown in Figure \ref{fig:iso_tri1} and its line graph is shown in Figure \ref{fig:iso_tri_line1}. Observe the bijective adjacency preserving mapping between the vertices of Figure \ref{fig:iso_tri1} and the triangles of Figure \ref{fig:iso_tri_line1}.
\begin{figure}[H]
    \begin{subfigure}{.5\textwidth}
      \centering
      \includegraphics[width=1\linewidth]{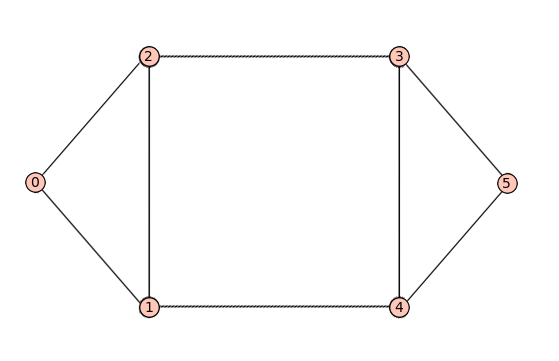}
      \caption{A graph satisfying (1), (2), (3), and (4)}
      \label{fig:iso_tri1}
    \end{subfigure}
    \begin{subfigure}{.5\textwidth}
      \centering
      \includegraphics[width=1\linewidth]{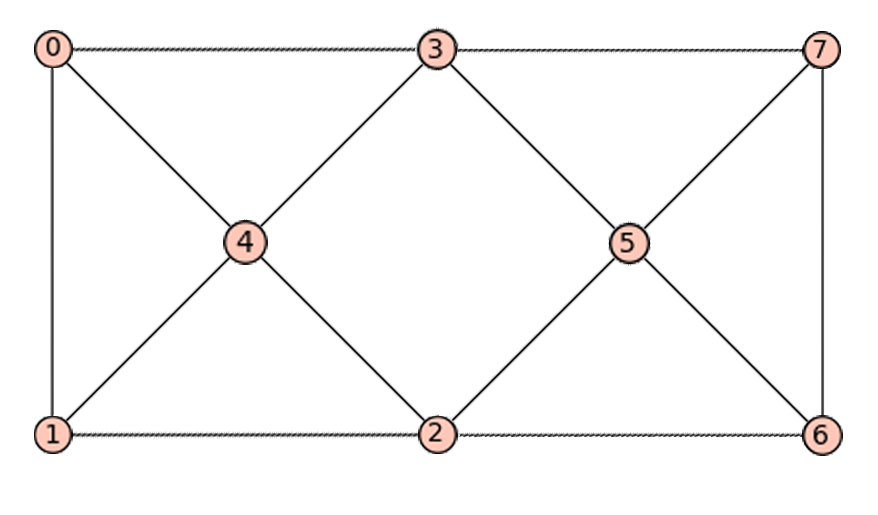}
      \caption{The line graph of (a)}
      \label{fig:iso_tri_line1}
    \end{subfigure}
    \label{fig:tri_subgraph1}
    \caption{}
    \end{figure}
\begin{proof}
    Begin by assuming (1), (2), (3), and (4). From these, observe that each triangle and each degree two vertex in $\Gamma$ must exist in a subgraph isomorphic to Figure \ref{fig:iso_tri}, which has line graph isomorphic to Figure \ref{fig:iso_tri_line}.
        \begin{figure}[H]
    \begin{subfigure}{.5\textwidth}
      \centering
      \includegraphics[width=.8\linewidth]{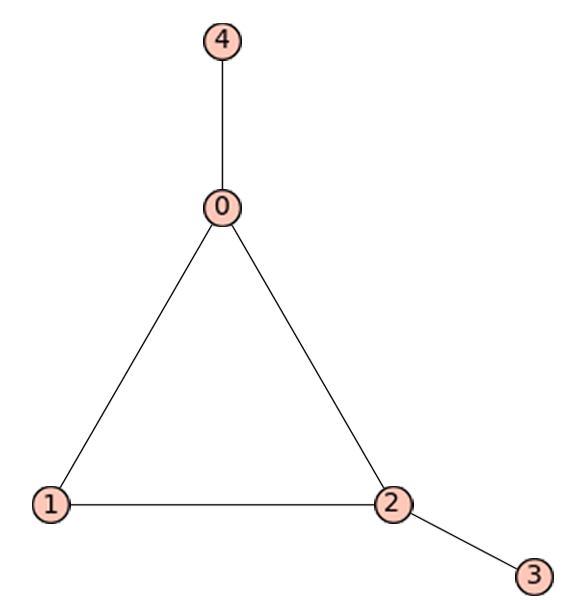}
      \caption{Subgraph of $\Gamma$}
      \label{fig:iso_tri}
    \end{subfigure}
    \begin{subfigure}{.5\textwidth}
      \centering
      \includegraphics[width=.75\linewidth]{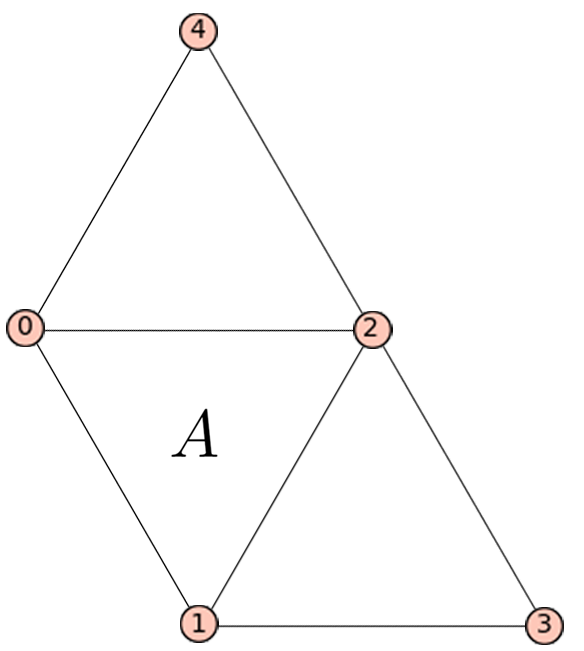}
      \caption{Subgraph of $L(\Gamma)$}
      \label{fig:iso_tri_line}
    \end{subfigure}
    \caption{}
    \end{figure}
    We define the mapping from the vertices of $\Gamma$ to the vertices of $C_3(L(\Gamma))$ by sending a degree $3$ vertex to the triangle it induces in $L(\Gamma)$ by Lemma \ref{lem:d(v)}, and sending a degree $2$ vertex to the triangle labeled $A$ in Figure \ref{fig:iso_tri_line}; the triangle in $L(\Gamma)$ induced by the unique triangle in $\Gamma$ containing the degree $2$ vertex as implied by (2). Since by Remark \ref{remark}, the only triangles in $L(\Gamma)$ are those created by a degree $3$ vertex or those induced by a triangle in $\Gamma$ with unique degree $2$ vertex, this mapping is bijective. We will now show it preserves adjacency. Let $u$ and $v$ be vertices in $\Gamma$. If $u$ and $v$ both have degree 3, the vertices they map to in $C_3(L(\Gamma))$ are adjacent if and only if $u$ and $v$ are adjacent by the same argument as in Theorem \ref{thm:4ci}. If both vertices have degree $2$, then they cannot be adjacent since by (2), this would imply they are in the same triangle, contradicting (3). And the vertices they map to in $C_3(L(\Gamma))$ are not adjacent since this would imply that two edges in the distinct triangles containing $u$ and $v$ share a vertex, contradicting (4). So now assume WLOG that $d(u)=2$ and $d(v)=3$. If $u$ and $v$ are adjacent in $\Gamma$, then Figure \ref{fig:iso_tri} is a subgraph of $\Gamma$ with $u$ as the vertex $1$ and $v$ as the vertex $0$. So $v$ is in the triangle in $\Gamma$ containing $u$ and the triangles they induce in $L(\Gamma)$ share the vertex $uv$. If the vertices that $u$ and $v$ map to in $C_3(L(\Gamma))$ are adjacent, then the triangles induced by the triangle containing $u$ and the edges incident to $v$ share some vertex in $L(\Gamma)$. This means the triangle containing $u$ in $\Gamma$ has some edge incident to $v$, implying $u$ and $v$ are adjacent. Thus $C_3(L(\Gamma)) \cong \Gamma$.\\
    For the converse, assume that $C_3(L(\Gamma)) \cong \Gamma$. The same argument as in Theorem \ref{thm:4ci} shows that $\Delta(\Gamma) \leq 3$. So for $1\leq i \leq 3$, let $d_i$ denote the number of vertices of degree $i$, and for $0 \leq j \leq 3$ let $t_j$ denote the number of triangles with $j$ vertices of degree $2$. From Remark \ref{remark}, the number of vertices in $C_3(L(\Gamma))$ is given by $d_3 + \sum_{j=0}^3 t_j$. First we will show $t_0 = t_3 = 0$. Since $\Gamma$ is connected, if there exists a triangle with $3$ vertices of degree $2$ then $\Gamma \cong K_3$ but $C_3(L(K_3)) \cong K_1$, a contradiction. If there exists a triangle with no vertices of degree $2$, then Figure \ref{fig:tri_t3} is a subgraph of $\Gamma$ where vertices 3, 4 and 5 may or may not be distinct. So Figure \ref{fig:tri_t3_line} is a subgraph of $L(\Gamma)$.
            \begin{figure}[H]
    \begin{subfigure}{.5\textwidth}
      \centering
      \includegraphics[width=.8\linewidth]{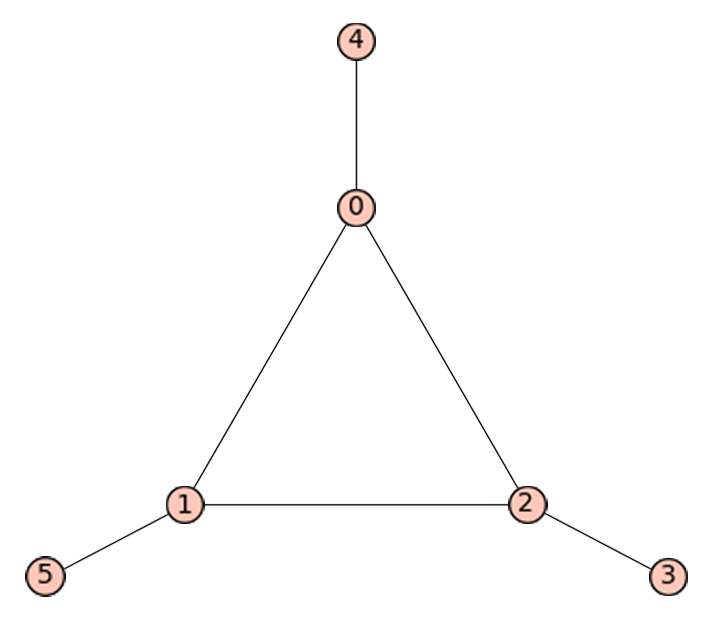}
      \caption{Subgraph of $\Gamma$}
      \label{fig:tri_t3}
    \end{subfigure}
    \begin{subfigure}{.5\textwidth}
      \centering
      \includegraphics[width=.8\linewidth]{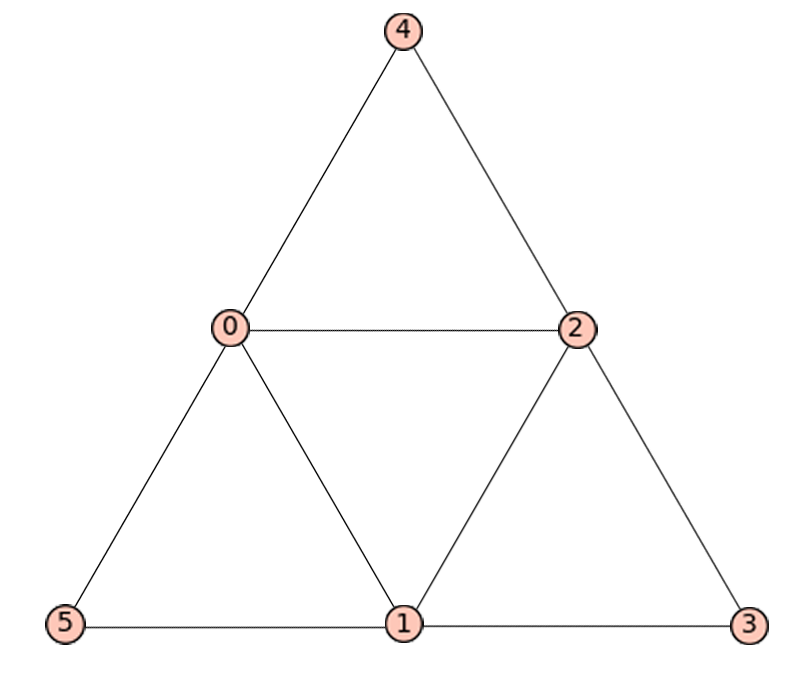}
      \caption{Subgraph of $L(\Gamma)$}
      \label{fig:tri_t3_line}
    \end{subfigure}
    \caption{}
    \end{figure}
    Figure \ref{fig:tri_t3_line} consists of four triangles with each pair sharing at least one vertex, so $K_4$ is a subgraph of $C_3(L(\Gamma)) \cong \Gamma$. Since $\Gamma$ is connected and $\Delta(\Gamma) \leq 3$, we must have $\Gamma \cong K_4$. But it can be verified that $C_3(L(K_4)) \not \cong K_4$, a contradiction.\\
    So we have $t_0=t_3=0$ implying the number of vertices in $C_3(L(\Gamma))$ is $d_3 + t_1 + t_2$ and since the number of vertices in $\Gamma$ is given by $d_1 + d_2 + d_3$, their isomorphism implies $t_1 + t_2 = d_1 + d_2$. A counting argument shows that $d_2 \geq t_1 + 2t_2$ implying 
    \begin{align*}
        t_1 + t_2 &= d_1 + d_2 \\
        t_1 + t_2 &\geq d_1 + t_1 + 2t_2 \\
        -t_2 &\geq d_1.
    \end{align*}
    Thus $t_2 = d_1 = 0$, since $0 \geq -t_2 \geq d_1 \geq 0$. This implies $(1)$ and $(3)$ and since $d_2 = t_1$, we have $(2)$ as well. To show (4), observe from (1), (2), and (3) that if two distinct triangles share vertices in $\Gamma$, Figure \ref{fig:line_tri_subgraph} must be a subgraph of $\Gamma$ and so Figure \ref{fig:disjoint_tri} is a subgraph of $L(\Gamma)$.
        \begin{figure}[H]
    \begin{subfigure}{.5\textwidth}
      \centering
      \includegraphics[width=.8\linewidth]{line_tri_subgraph.png}
      \caption{Subgraph of $\Gamma$}
      \label{fig:line_tri_subgraph}
    \end{subfigure}
    \begin{subfigure}{.5\textwidth}
      \centering
      \includegraphics[width=.6\linewidth]{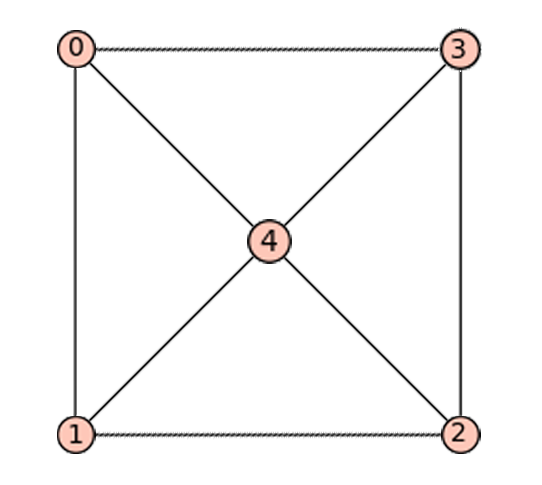}
      \caption{Subgraph of $L(\Gamma)$}
      \label{fig:disjoint_tri}
    \end{subfigure}
    \caption{}
    \end{figure}
    Figure \ref{fig:disjoint_tri} consists of four triangles all sharing vertex 5, so $K_4$ is a subgraph of $C_3(L(\Gamma))$. But as shown above, this leads to a contradiction so we get (4).
 \end{proof}
\subsection{Spectrum of Clique Graphs}
\tab Let $m$ be the number of $\omega$-cliques in an $\omega$-clique regular graph $\Gamma$ so that $m\binom{\omega}{2}$ is the number of edges in $\Gamma$. Let $A_C$ be the adjacency matrix of $C_\omega(\Gamma)$ with vertices indexed $(c_1, \ldots , c_{m})$ and $A_L$ be the adjacency matrix of $L(\Gamma)$ with vertices indexed $(e_1, \; \ldots, \;e_{m\binom{\omega}{2}})$ such that edges $\{e_{(i-1)\binom{\omega}{2} +1},\; \ldots, \; e_{i\binom{\omega}{2}}\}$ constitute clique $c_i$ for all $1 \leq i \leq m$, i.e. clique $c_1$ contains edges $e_1$ through $e_{\binom{\omega}{2}}$, clique $c_2$ contains edges $e_{\binom{\omega}{2}+1}$ through $e_{2\binom{\omega}{2}}$ and so on.
Now define $\tilde{A} = 6\binom{\omega}{3}I_{m} + (\omega -1)^2A_C$ and
define $\varphi : \mathbb{R}^{m} \rightarrow \mathbb{R}^{m\binom{\omega}{2}}$ by 
\[\varphi\left(\begin{bmatrix}a_1\\a_2\\\vdots\\a_{m}
\end{bmatrix}\right)=\begin{bmatrix}\Vec{a}_1\\\Vec{a}_2\\\vdots\\\Vec{a}_{m}\end{bmatrix} \text{ where } \:\ \Vec{a}:=
    \begin{bmatrix}a\\\vdots\\a\end{bmatrix} \in \mathbb{R}^{\binom{\omega}{2}}.\]
So the $\varphi$ function expands a column vector by a factor of $\binom{\omega}{2}$. For example, if $m=2$ and $\omega = 3$ then $\varphi\left(\begin{bmatrix}1\\2\end{bmatrix}\right)=\begin{bmatrix}1 &1&1&2&2&2\end{bmatrix}^T$.
\begin{lem} \label{lem:phi}
If $\Gamma$ is $\omega$-clique regular with $\tilde{A},\; \varphi$, and $A_L$ as above, then for all $v \in \mathbb{R}^{m}$, \[v^T \tilde{A}v = \varphi(v)^T A_L \varphi(v).\]
\end{lem}
\begin{proof}
    Let $t_i$ be the standard basis vector of $\mathbb{R}^{m}$ with a $1$  in the $i^{\text{th}}$ position for $1 \leq i \leq m$.
    It is sufficient to show that for all $t_i$ and $t_j$
    \[t_i^T \tilde{A}t_j = \varphi(t_i)^T A_L \varphi(t_j).\]
    Observe that $\varphi(t_i)^T A_L \varphi(t_j)$ is the $1\times1$ matrix with entry as the sum of the values of the sub-matrix of $A_L$, $S_{ij}$, created by taking the rows corresponding to edges in $c_i$ and columns corresponding to edges in $c_j$.\\
    Similar to above, it is clear that $t_i^T \tilde{A}t_j$ is the $1\times 1$ matrix with entry $(\tilde{A})_{ij}$.\\
    Now for the first case $i=j$. Let $e_k$ be an edge in $\Gamma$ with endpoints $x$ and $y$ in $c_i$. So $e_k$ corresponds to a row in $S_{ii}$. The vertex $x$ is incident to $\omega -1$ edges in $c_i$ one of which is $e_k$ and the same for vertex $y$. So $e_k$ is incident to $2(\omega-2)$ edges in $c_i$ implying the sum of the $e_k$ row in $S_{ii}$ is $2(\omega-2)$. Since there are $\binom{\omega}{2}$ rows in $S_{ii}$, the total sum of entries in $S_{ii}$ is $2(\omega-2)\binom{\omega}{2} = 6\binom{\omega}{3}$.\\So $t_i^T \tilde{A}t_i =\left[6\binom{\omega}{3}\right] = \varphi(t_i)^T A_L \varphi(t_i)$.\\
    Next for the case when $c_i$ is adjacent to $c_j$. Let $x$ be the unique vertex in $\Gamma$ that $c_i$ and $c_j$ share and let $e_k$ be an edge incident to $x$ with other endpoint in $c_i$. Then $e_k$ is incident to the $\omega-1$ edges of $x$ with other endpoints in $c_j$ implying the sum of the $e_k$ row in $S_{ij}$ is $\omega-1$. Since there are $\omega-1$ edges incident to $x$ with other endpoint in $c_i$, and since every edge in $c_i$ that isn't incident to $x$ is not incident to any edges in $c_j$, the total sum of the entries in $S_{ij}$ is $(\omega-1)^2$. So $t_i^T \tilde{A}t_j =\left[(\omega-1)^2\right] = \varphi(t_i)^T A_L \varphi(t_j)$.\\
    Finally for the last case when $i \neq j$ and $c_i$ is not adjacent to $c_j$. Let $e_k$ be an edge in $c_i$. Then since $c_i$ and $c_j$ don't share common vertices, $e_k$ is not incident to any edges in $c_j$ implying that the total sum of the entries in $S_{ij}$ is 0.\\So $t_i^T \tilde{A}t_j =\left[0\right] = \varphi(t_i)^T A_L \varphi(t_j)$.
\end{proof}

\begin{thm} \label{thm:bounds}
    Suppose $\Gamma$ is $\omega$-clique regular and the eigenvalues of $L(\Gamma)$ are $\mu_1 \leq \cdots \leq \mu_{m\binom{\omega}{2}}$. Then for each eigenvalue $\lambda$ of $C_\omega(\Gamma)$,
    \[\frac{\omega}{\omega - 1}\left(\frac{\mu_1}{2} -\omega +2\right) \leq \lambda \leq \frac{\omega}{\omega - 1}\left(\frac{\mu_{m\binom{\omega}{2}}}{2} -\omega +2\right).\]
\end{thm}
\begin{proof}
    For this proof $|t|$ will denote the euclidean norm of vector $t$. Let $\lambda$ be an eigenvalue of $A_C$ with eigenvector $u$. Clearly, $A_C$ and $\tilde{A} = 6\binom{\omega}{3}I_{m} + (\omega -1)^2A_C$ share the same eigenvectors and $\eta= 6\binom{\omega}{3}+(\omega-1)^2\lambda$ is an eigenvalue of $\tilde{A}$ with the eigenvector $u$. Now define $v = \frac{u}{|u|\sqrt{\binom{\omega}{2}}}$ so that $|v|^2=\frac{1}{\binom{\omega}{2}}$ and $v$ is an eigenvector of $\tilde{A}$ with eigenvalue $\eta$. Notice that 
    \[\varphi(v)^T\varphi(v)=\left[|\varphi(v)|^2\right]=\left[\binom{\omega}{2}|v|^2\right]=[1]=I_1.\]
    By the Eigenvalue Interlacing Theorem, \cite[pg. 26]{haemers} this implies that the eigenvalue of the $1 \times 1$ matrix $\varphi(v)^T A_L \varphi(v)$ interlaces the eigenvalues of $A_L$. By Lemma \ref{lem:phi}, 
    \[\varphi(v)^T A_L \varphi(v)  = v^T \tilde{A}v = \eta v^T v =\left[\eta|v|^2\right]=\left[\frac{\eta}{\binom{\omega}{2}}\right].\] So the eigenvalue $\frac{\eta}{\binom{\omega}{2}}$ interlaces the eigenvalues of $A_L$, $\mu_1 \leq \frac{\eta}{\binom{\omega}{2}} \leq \mu_{m\binom{\omega}{2}}$. \\Since $\eta=6\binom{\omega}{3}+(\omega - 1)^2\lambda$, it follows that
     \[\frac{1}{(\omega - 1)^2}\left(\binom{\omega}{2} \mu_1-6\binom{\omega}{3}\right) \leq \lambda \leq \frac{1}{(\omega - 1)^2}\left(\binom{\omega}{2} \mu_{m\binom{\omega}{2}}-6\binom{\omega}{3}\right).\]
    Equivalently,
    \[\frac{\omega}{\omega - 1}\left(\frac{\mu_1}{2} -\omega +2\right) \leq \lambda \leq \frac{\omega}{\omega - 1}\left(\frac{\mu_{m\binom{\omega}{2}}}{2} -\omega +2\right).\]
\end{proof}

\begin{figure}[H]
    \centering
    \includegraphics[width=0.5\linewidth]{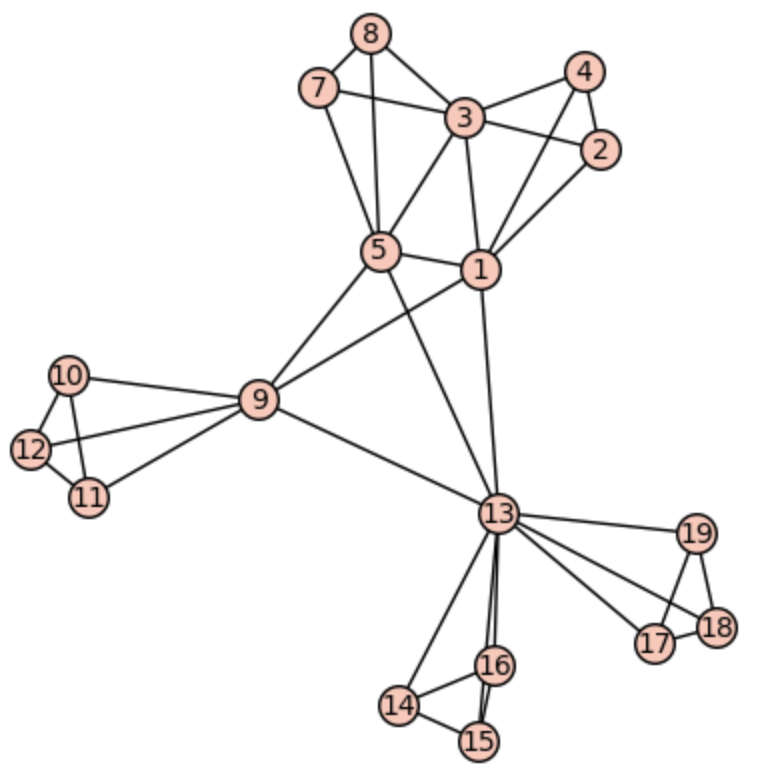}
    \caption{}
    \label{fig:7}
\end{figure}

\begin{exmp}
    Let $\Gamma$ be the graph in Figure \ref{fig:7}. Then $C_4(\Gamma)$ will have 6 vertices and 7 edges and since $\Gamma$ is 4-clique regular then we can use Theorem \ref{thm:bounds} which tells us that the eigenvalues of $C_4(\Gamma$) will be in the range $[-4,3.4782]$.
    Using SAGE we find that the largest eigenvalue of $C_4(\Gamma)$ is 2.7092 and the smallest eigenvalue is -1.9032 which both fall within this range.
\end{exmp}

\begin{cor}
    If $\Gamma$ is $\omega$-clique regular, connected and not a complete graph, then for $\mu_{m\binom{\omega}{2}}$ the largest eigenvalue of $L(\Gamma)$, $2\omega -4 < \mu_{m\binom{\omega}{2}}$.
\end{cor}
\begin{proof}
    Since $\Gamma$ is connected and not a complete graph, the $\omega$-clique graph of $\Gamma$ contains an edge. So there exist nonzero eigenvalues of $A_C$ and since they sum to 0, the largest eigenvalue of $A_C$ is greater than 0. This implies by Theorem \ref{thm:bounds} that $0 < \frac{\omega}{\omega - 1}\left(\frac{\mu_{m\binom{\omega}{2}}}{2} -\omega +2\right)$, from which it follows $2\omega -4 <\mu_{m\binom{\omega}{2}}$.
\end{proof}

Note that it is similar to show $\mu_1 < 2\omega -4$ but this result is redundant since $\mu_1 <0$ and $2\omega -4 \geq 0$ for $\omega \geq 2$.
\begin{cor} \label{cor:eigens}
    If $\Gamma$ is $\omega$-clique regular then for each eigenvalue $\lambda$ of $C_\omega(\Gamma)$, \[-\omega \leq \lambda \leq \omega\left(\frac{\Delta(\Gamma)}{\omega-1}-1\right)\] where $\Delta(\Gamma)$ denotes the largest degree of $\Gamma$.
\end{cor} 
\begin{proof}
    Let $\mu_1 \leq \cdots \leq \mu_{m\binom{\omega}{2}}$ be the eigenvalues of $L(\Gamma)$. Two classical results in spectral graph theory are that for any line graph, $-2 \leq \mu_1$ and $\mu_{m\binom{\omega}{2}} \leq \Delta(L(\Gamma))\leq 2(\Delta(\Gamma)-1)$. From these we get 
    \[ -\omega \leq \frac{\omega}{\omega - 1}\left(\frac{\mu_1}{2} -\omega +2\right)\] and \[ \frac{\omega}{\omega - 1}\left(\frac{\mu_{m\binom{\omega}{2}}}{2} -\omega +2\right) \leq \omega\left(\frac{\Delta(\Gamma)}{\omega-1}-1\right).\]
    So the result follows from Theorem \ref{thm:bounds}.
\end{proof}
If $\Gamma$ is an $\omega$-clique regular graph and also $k$-regular, then we can show that the characteristic polynomial of $C_\omega(\Gamma)$ is a function of the characteristic polynomial of $\Gamma$. Recall that $p(\Gamma; \lambda)$ denotes the characteristic polynomial of $\Gamma$'s adjacency matrix and that the roots of this polynomial are the eigenvalues of $\Gamma$. Let $n$ be the number of vertices in $\Gamma$ and $m=\frac{nk}{\omega (\omega -1)}$ be the number of $\omega$-cliques in $\Gamma$, i.e. the number of vertices in $C_\omega(\Gamma)$. Let $A$ be the adjacency matrix of $\Gamma$ with vertices indexed $(v_1, \ldots , v_n)$, and let $A_C$ be the adjacency matrix of $C_\omega(\Gamma)$ with $\omega$-cliques indexed $(c_1, \ldots , c_m)$. We will define the \textbf{$\omega$-clique incidence matrix} of $\Gamma$ as an $n\times m$ matrix $R$ such that the rows are indexed by the vertices in $\Gamma$ in the same order as in $A$, and the columns are indexed by the $\omega$-cliques in $\Gamma$ in the same order as in $A_C$. The entries of $R$ are defined,
\[(R)_{ij} = 
    \begin{cases}
        1 & \text{if } v_i \in c_j\\
        0 & \text{otherwise.}
    \end{cases}\]
Recall also that the \textbf{degree matrix} of $\Gamma$ is defined as $D=\text{diag}(d(v_1), \ldots , d(v_n))$.

\begin{lem} \label{lem:blockmat}
Suppose that $\Gamma$ is $\omega$-clique regular with $\omega$-clique incidence matrix $R$ and degree matrix $D$. Then:\vspace{0.2cm} \\
           \hspace*{0.5cm} (1) $R^TR=A_C + \omega I_m$ and  \\
           \hspace*{0.5cm} (2) $RR^T = A + \frac{1}{\omega -1}D$. \vspace{0.2cm}
\end{lem}
\begin{proof}
    $(1)$ We have
    \[(R^TR)_{ij} = \sum_{l=1}^n (R)_{li}(R)_{lj},\] 
    from which it is clear that $(R^TR)_{ij}$ is the number of vertices that are in both clique $c_i$ and $c_j$. If $i=j$, then the sum will equal $\omega$, which is the number of vertices in $c_i$. So $(R^TR)_{ii} = \omega =(A_C + \omega I_m)_{ii}$. If $c_i$ and $c_j$ are adjacent, then they share one vertex in common. So $(R^TR)_{ij} = 1 = (A_C + \omega I_m)_{ij}$. \\Otherwise $(R^TR)_{ij}=0=(A_C + \omega I_m)_{ij}$.\\
    $(2)$ Similarly, we have that $(RR^T)_{ij}$ is the number of cliques that $v_i$ and $v_j$ commonly belong to. If $i=j$, then the sum will equal the number of cliques that contain $v_i$. Since $v_i$ has degree $d(v_i)$, and each clique with $v_i$ contains $\omega -1$ edges adjacent to $v_i$ unique to that clique, $\frac{d(v_i)}{\omega -1}$ is the number of cliques that $v_i$ belongs to. So $(RR^T)_{ij} = \frac{d(v_i)}{\omega -1} = (A + \frac{1}{\omega -1}D)_{ij}$. If $v_i$ and $v_j$ are adjacent, then the edge between them belongs to a unique clique so $(RR^T)_{ij} = 1 = (A + \frac{1}{\omega -1}D)_{ij}$. \\Otherwise $(RR^T)_{ij} = 0 =(A + \frac{1}{\omega -1}D)_{ij}$.
\end{proof}
\begin{thm} \label{thm:eigen}
    If $\Gamma$ is $k$-regular and $\omega$-clique regular, then 
    \[p(C_\omega(\Gamma); \lambda)=(\lambda+\omega)^{m-n}p\left(\Gamma;\lambda+\omega-\frac{k}{\omega - 1}\right)\]
    where $m=\frac{nk}{\omega(\omega-1)}$.
\end{thm}
\begin{proof}
    Define two square block matrices with $n+m$ rows and columns as follows with $R$ the $\omega$-clique incidence matrix of $\Gamma$,
    \[U = \begin{bmatrix}(\lambda+\omega) I_n & -R \\0 & I_m \\\end{bmatrix}, \hspace{1cm}
    V = \begin{bmatrix}I_n & X \\R^T & (\lambda+\omega) I_m \\\end{bmatrix}.\]
Then, 
\[UV = \begin{bmatrix}(\lambda+\omega) I_n - RR^T & 0 \\R^T & (\lambda+\omega) I_m \\\end{bmatrix}, 
    VU = \begin{bmatrix}(\lambda+\omega) I_n & 0 \\(\lambda+\omega) R^T & (\lambda+\omega) I_m - R^TR \\\end{bmatrix}.\]
Since $\text{det}(UV)=\text{det}(VU)$, it follows that
\begin{align*}
    (\lambda+\omega)^m \text{det}((\lambda+\omega) I_n - RR^T) & =(\lambda+\omega)^n \text{det}((\lambda+\omega) I_m - R^TR) \\ 
    (\lambda+\omega)^{m-n} \text{det}((\lambda+\omega) I_n - RR^T) & =\text{det}((\lambda+\omega) I_m - R^TR).
\end{align*}
Using the above equality and Lemma \ref{lem:blockmat} we get,
\begin{align*}
    p(C_\omega(\Gamma);\lambda) & = \text{det}(\lambda I_m - A_C) \\ 
    & = \text{det}((\lambda + \omega)I_m - R^TR)  \\
    & = (\lambda + \omega)^{m-n}\text{det}((\lambda + \omega)I_n - RR^T) \\
    & = (\lambda + \omega)^{m-n}\text{det}\left(\left(\lambda + \omega - \frac{k}{\omega -1}\right)I_n - A\right)\\
    & = (\lambda + \omega)^{m-n}p\left(\Gamma ; \lambda + \omega - \frac{k}{\omega -1}\right).
\end{align*}
\end{proof}
Note that for the previous lemma and theorem, taking $\omega =2$ gives the proof from Biggs \cite[p.18-19]{biggs} for the characteristic polynomial of the line graph of a regular graph. Here we generalized this proof for the $\omega$-clique graph of any graph that is regular and $\omega$-clique regular for any $\omega \geq 2$.
\begin{remark} \label{rem:eigen}
    Recall that the eigenvalues of a matrix are exactly the roots of the matrix's characteristic polynomial. So Theorem \ref{thm:eigen} implies that if the spectrum of a $k$-regular and $\omega$-clique regular graph is
    \[\lambda_1^{a_1}, \ldots , \lambda_r^{a_r}, k^c,\]
    then the spectrum of its $\omega$-clique graph is 
    \[-\omega^{m-n}, \left(\frac{k}{\omega -1} + \lambda_1 -\omega \right)^{a_1}, \ldots, \left(\frac{k}{\omega -1} + \lambda_r -\omega \right)^{a_r}, \left(\frac{k}{\omega -1} + k -\omega \right)^c\]
    where $n$ is the number of vertices and $m=\frac{nk}{\omega(\omega -1)}$ is the number of $\omega$-cliques.
\end{remark}
\begin{cor} \label{cor:lowbound}
    Suppose $\Gamma$ is $k$-regular and $\omega$-clique regular where the smallest eigenvalue $\lambda_1$ has multiplicity $a_1$. Then $\lambda_1 \geq \frac{-k}{\omega-1}$ and if $k < \omega(\omega -1)$ as well, then $\lambda_1 = \frac{-k}{\omega-1}$ and $a_1 \geq n-\frac{nk}{\omega(\omega -1)}.$
\end{cor}
\begin{proof}
    From Remark \ref{rem:eigen}, $\frac{k}{\omega -1}+\lambda_1 -\omega$ is an eigenvalue of $C_\omega(\Gamma)$ so from Corollary \ref{cor:eigens} $-\omega \leq \frac{k}{\omega -1} +\lambda_1 -\omega$ implies $\lambda_1 \geq \frac{-k}{\omega -1}$. If $k < \omega(\omega-1)$ then \linebreak$m-n=\frac{nk}{\omega(\omega -1)}-n <0$. Since this is the exponent of the $(\lambda + \omega)$ term in the characteristic polynomial of $C_\omega(\Gamma)$ by Theorem \ref{thm:eigen}, there must be the same term in $p(\Gamma; \lambda +\omega - \frac{k}{\omega -1})$ with exponent greater than or equal to $n-\frac{nk}{\omega(\omega -1)}$. By above, this term must come from the smallest eigenvalue so we get $\lambda_1 = \frac{-k}{\omega -1}$ with multiplicity $a_1 \geq n - \frac{nk}{\omega(\omega -1)}$.
\end{proof}

\subsection{Strongly Regular Graphs}
\tab There is much overlap between clique regular graphs and strongly regular graphs. The following theorem will determine when the $\omega$-clique graph of an $\omega$-clique regular and strongly regular graph is also strongly regular. 
\begin{thm} \label{thm:srg}
    Suppose $\Gamma$ is $\omega$-clique regular and a non-boring srg$(n,k,\lambda, \mu)$ with spectrum $k^1,r^f,s^g$ where $r>s$. Then the $\omega$-clique graph of $\Gamma$ is strongly regular if and only if $s = \frac{-k}{\omega-1}$ or $k = \omega(\omega -1)$. If so, $C_\omega(\Gamma)$ has parameters
    $$\text{srg}\left( \frac{nk}{\omega(\omega -1)},\; \omega \left(\frac{k}{\omega-1}-1\right),\; \lambda^*,\; \mu^* \right)$$
    and if $\lambda = \omega - 2$, then $\lambda^* = \frac{k}{\omega -1}-2$ and $\mu^*=\mu +\omega-\frac{k}{\omega-1}$.
\end{thm}
Note that $\lambda^*$ and $\mu^*$ are forced by the spectrum of $C_\omega(\Gamma)$ and can be always be derived from it, but if $\Gamma$ is a regular clique assembly (equivalent to $\lambda = \omega - 2$ by Theorem \ref{thm:rca}) we can derive simpler formulas for the srg parameters.
\begin{proof}
    From Remark \ref{rem:eigen} the spectrum of $C_\omega(\Gamma)$ is \[-\omega^{m-n},\; \left(\frac{k}{\omega-1}+k-\omega \right)^1,\; \left(\frac{k}{\omega-1}+r-\omega\right)^f,\;  \left(\frac{k}{\omega-1}+s-\omega\right)^g \mbox{(*)}\] where $m=\frac{nk}{\omega(\omega-1)}$. Since $C_\omega(\Gamma)$ is $\omega(\frac{k}{\omega-1}-1)$-regular, then it is sufficient to show that the adjacency matrix of  $C_\omega(\Gamma)$ has 3 or less distinct eigenvalues if and only if $s = \frac{-k}{\omega-1}$ or $k = \omega(\omega -1)$. If $s = \frac{-k}{\omega-1}$, then $\left(\frac{k}{\omega-1}+s-\omega\right) =-\omega$ and if $k=\omega(\omega-1)$ then $m=n$  implies the multiplicity of $-\omega$ is 0. So regardless, $C_\omega(\Gamma)$ has no more than 3 distinct eigenvalues.\\Conversely, assume $C_\omega(\Gamma)$ has 3 or less distinct eigenvalues. This implies that one of the multiplicities in the spectrum $(*)$ equals 0 or two of the eigenvalues are equal. If the first case, then $m-n$ must equal 0 since $1,f,g \neq 0$. Then $m=\frac{nk}{\omega(\omega-1)}=n$ implies $k=\omega(\omega-1)$. If two of the eigenvalues of are equal, we will show $s=\frac{-k}{\omega-1}$. If any of the latter 3 eigenvalues in $(*)$ are equal, it would imply that $k=r$, $k=s$ or $r=s$, a contradiction. So $-\omega$ is equal to one of the other eigenvalues implying $k, r$ or $s$ must be equal to $\frac{-k}{\omega -1}$. From Corollary \ref{cor:lowbound}, we have $k>r>s \geq \frac{-k}{\omega - 1}$ implying $s=\frac{-k}{\omega -1}$.\\
    Now assume that $C_\omega(\Gamma)$ is strongly regular and $\lambda = \omega -2$. For the case $\omega=2$ or $3$, the result $\lambda^*=\frac{k}{\omega-1}-2$ is shown in \cite[pg. 304]{guest}. So assume $\omega \geq 4$.\\
    Let $c_1$ and $c_2$ be adjacent vertices in $C_\omega(\Gamma)$ (that is,  $\omega$-cliques in $\Gamma$) and let $x \in \Gamma$ be the unique vertex that $c_1$ and $c_2$ share. Then $c_1$ and $c_2$ are commonly adjacent to all the other cliques containing $x$. So $\lambda^* \geq \frac{k}{\omega-1} -2$. Suppose for contradiction there exists another clique $c_i$ commonly adjacent to $c_1$ and $c_2$ such that $x$ is not in $c_i$. Let $y_1$ and $y_2$ respectively be the unique common vertex of $c_1$ and $c_i$, and $c_2$ and $c_i$. Recall that $\lambda = \omega -2$ implies $\Gamma$ is a regular clique assembly and every maximal clique has order $\omega$ by Theorem \ref{thm:rca}. So the clique $\{x,y_1,y_2\}$ is contained in an $\omega$-clique not equal to $c_1$ since $c_1$ does not contain $y_2$. Thus the edge $xy_1$ is in two different $\omega$-cliques, a contradiction. So $\lambda^* = \frac{k}{\omega -1} -2$.\\
    For $\mu^*$ recall that for any srg, $\lambda - \mu = r+s$. So we get
    \begin{align*}
        \lambda^* - \mu^* &= r^*+s^* \\
        \mu^* &= \lambda^* -r^*-s^* \\
        &= \left(\frac{k}{\omega-1}-2\right)-\left(\frac{k}{\omega -1} +r -\omega\right) -\left(\frac{k}{\omega -1} +s -\omega\right) \\
        &= 2\omega - (r+s)-2-\frac{k}{\omega-1} \\
        &= 2\omega - (\lambda - \mu)-2-\frac{k}{\omega-1}\textbf{}\\
        &= 2\omega - (\omega-2 - \mu)-2-\frac{k}{\omega-1}\\
        &= \mu +\omega- \frac{k}{\omega-1}.
    \end{align*}
\end{proof}

\begin{cor} \label{cor:kww}
    Suppose $\Gamma$ is $\omega$-clique regular and a non-boring srg$(n,k,\lambda, \mu)$. Then $\Gamma$ and $C_\omega(\Gamma)$ are srgs with the same parameters if and only if $k=\omega(\omega-1)$.
\end{cor}
\begin{proof}
    If $k=\omega(\omega-1)$, then it follows that $\Gamma$ and $C_\omega(\Gamma)$ have the same spectrum, and it is easily provable that strongly regular graphs with the same spectrum have the same parameters. Conversely if $\Gamma$ and $C_\omega(\Gamma)$ have the same parameters then $\frac{nk}{\omega(\omega -1)}=n$ which implies $k=\omega(\omega -1)$.
\end{proof}

We will now give a necessary condition on strongly regular graphs that are also regular clique assemblies.
\begin{lem}
    If $\Gamma$ is an srg$(n,k,\omega -2, \mu)$ and is $\omega$-clique regular, \linebreak then $k \geq \mu(\omega -1)$.
\end{lem}
\begin{proof}
    If $\omega=2$, then the theorem is shown since $k\geq \mu $ for all strongly regular graphs. So assume $\omega \geq 3$. Recall that $\Gamma$ is a regular clique assembly by Theorem \ref{thm:rca}. So from \cite{guest} Proposition 1, the neighborhood of each vertex in $\Gamma$ is a disjoint union of $K_{\omega -1}$'s with $\frac{k}{\omega -1}$ components. Let vertices $v$ and $u$ be nonadjacent in $\Gamma$. Suppose $x$ and $y$ are distinct vertices in $N(v)\cap N(u)$. If $x$ and $y$ are in the same $K_{\omega -1}$ component of $N(v)$, then since $\Gamma$ is a regular clique assembly, the clique $\{x,y,u\}$ is contained in an $\omega$-clique containing $u$ and the clique $\{x,y,v\}$ is contained in a different $\omega$-clique containing $v$. But this is a contradiction since the edge $\{x,y\}$ is in two different $\omega$-cliques. So $x$ and $y$ must be in distinct $K_{\omega-1}$ components in $N(v)$ showing there exists an injective mapping from $N(v)\cap N(u)$ to the disjoint $K_{\omega-1}$ components in $N(v)$. This implies $\frac{k}{\omega -1} \geq \mu$ from which the result follows.
\end{proof}
\section{Examples and Applications}
\subsection{Orthogonal Array Block Graphs}
\tab An \textbf{orthogonal array}, denoted $OA(n,m)$, is an $n^2 \times m$ array with entries from an $n$-element set with the property that the rows of any $n^2 \times 2$ sub array consist of all $n^2$ possible pairs exactly once. The \textbf{block graph} of an orthogonal array is the graph with vertices as the $1 \times m$ row vectors of the $OA(n,m)$, where two vectors are adjacent if and only if they have nonempty intersection. This means they share the same entry in exactly one position, since by the construction of the orthogonal array no two vectors can share the same entry in more than one position.
    \begin{figure}[H]
    \begin{subfigure}{.5\textwidth}
      \centering

\begin{tabular}{|c|c|c|}
\hline
$1$ & $1$ & $1$ \\
\hline
$1$ & $2$ & $2$ \\
\hline
$1$ & $3$ & $3$ \\
\hline
$2$ & $1$ & $2$ \\
\hline
$2$ & $2$ & $3$ \\
\hline
$2$ & $3$ & $1$ \\
\hline
$3$ & $1$ & $3$ \\
\hline
$3$ & $2$ & $1$ \\
\hline
$3$ & $3$ & $2$ \\
\hline
\end{tabular}

      \caption{An OA$(3,3)$ and}
    
    \end{subfigure}
    \begin{subfigure}{.5\textwidth}
      \centering
      \includegraphics[width=0.7\linewidth]{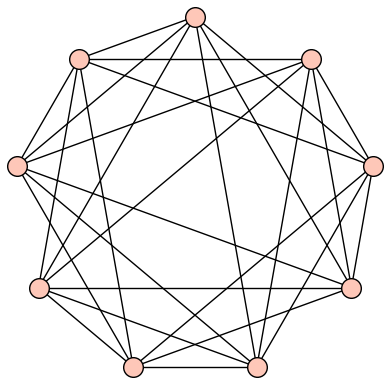}
      \caption{its Block Graph}
      
    \end{subfigure}
    \label{fig:oa_block}
    \caption{}
    \end{figure}
A clique in an orthogonal array block graph is called a \textbf{canonical clique}, denoted $S_{ri}$, if every vector in the clique shares the same entry $i$ in the same column $r$. Clearly every maximal clique of this form has order $n$.
\begin{lem} \label{lem: OA}
    If $\Gamma$ is the block graph of  an $OA(n,m)$ and $n>(m-1)^2$, then every clique of order $n$ is a canonical clique.
\end{lem}
The proof of this lemma comes from \cite[pg. 99]{godsil} Corollary 5.5.3.
\begin{thm} \label{thm:OA}
    If $\Gamma$ is the block graph of  an $OA(n,m)$ and $n > (m-1)^2$, then $\Gamma$ is $n$-clique regular and $C_n(\Gamma)$ is isomorphic to the complete $m$-partite graph where each independent set of vertices has order $n$.
\end{thm}
\begin{proof}
    If two vectors $v$ and $u$ are adjacent in $\Gamma$, then they share the same entry $i$ in some column $r$. This is the only entry that the two vectors share by the construction of the orthogonal array. So the only canonical clique the edge $vu$ is in is $S_{ri}$ and by Lemma \ref{lem: OA} this is the only $n$-clique the edge is in.\\
    Also by Lemma \ref{lem: OA}, the canonical cliques of $\Gamma$ form all the vertices in $C_n(\Gamma)$. Notice that if $i\neq j$ then for all columns $r$, $S_{ri}$ and $S_{rj}$ have an empty intersection so are not adjacent in $C_n(\Gamma)$. Also notice that by the definition of an orthogonal array, if $r\neq t$ then for all $i$ and $j$, $S_{ri}$ and $S_{tj}$ share the vector with entry $i$ in column $r$ and entry $j$ in column $t$ and so are adjacent in $C_n(\Gamma)$. So there are $m$ independent sets of order $n$ in $C_n(\Gamma)$ and each vertex in each set is adjacent to every other vertex of every other independent set. Thus, $C_n(\Gamma)$ is isomorphic to the complete $m$-partite graph where each independent set of vertices has order $n$.
\end{proof}
A \textbf{square rook graph} is the graph with vertices as the points in a square grid where two points are adjacent if and only if they are in the same row or column.
\begin{thm} \label{thm:rook}
    If $\Gamma$ is the square rook graph on $n^2$ vertices, then $\Gamma$ is $n$-clique regular and $C_n(\Gamma) \cong K_{n,n}$.
\end{thm}
\begin{proof} 
    Since the square rook graph on $n^2$ vertices is the block graph of an $OA(n,2)$, the result follows from Theorem \ref{thm:OA}.\\
    The result also follows from Theorems \ref{thm:4cr} - \ref{thm:3ci} since the square rook graph on $n^2$ vertices is the line graph of $K_{n,n}$ which is $n$-regular and triangle free.\\
    It also follows using the graph's spectrum since the square rook graph on $n^2$ vertices is an srg$(n^2, 2(n-1), n-2, 2)$ with spectrum \[2(n-1)^1,(n-2)^{2(n-1)}, -2^{(n-1)^2}.\] Using Remark \ref{rem:eigen}, the spectrum of $C_n(\Gamma)$ is \[n^1, 0^{2(n-1)}, -n^1\] and a classical result in spectral graph theory states this spectrum uniquely determines that $C_n(\Gamma) \cong K_{n,n}$.
\end{proof}

\subsection{Triangular Graphs}
\tab A \textbf{triangular graph}, denoted $T_n$, is the line graph of the complete graph on $n$ vertices, $K_n$. 
\begin{thm}
    If $n = 3$ or $n>4$, then $T_n$ is $(n-1)$-clique regular and $C_{n-1}(T_n) \cong K_n$.\\
\end{thm}
\begin{proof}
    This follows from Theorems \ref{thm:4cr} and \ref{thm:4ci} since $K_n$ is $(n-1)$-regular. Note that $T_4$ is not included since $K_4$ is $3$-regular but not triangle free.\\
    Alternatively we can prove this using the spectrum of $T_n$ which is \[2(n-2)^1,\;(n-4)^{n-1},\;-2^{\frac{n(n-3)}{2}}\]
    because $T_n$ is an srg$(\frac{1}{2}n(n - 1),\; 2(n - 2),\; n - 2,\; 4)$.
    And since $T_n$ is also $(n-1)$-clique regular and $2(n-2)$-regular, Remark \ref{rem:eigen} gives the spectrum of $C_{n-1}(T_n)$ as \[ (n-1)^1,\; -1^{n-1}.\]
    A well known result in spectral graph theory is that the only graph with this spectrum is the complete graph on $n$ vertices. So $C_{n-1}(T_n) \cong K_n$.
\end{proof}

\subsection{Generalized Quadrangle Collinearity Graphs}
\tab A \textbf{Generalized Quadrangle} $GQ(s,t)$ is a point-line incidence structure satisfying the following properties for some $s,t \geq 1$: \begin{itemize}
    \item every line has $s+ 1$ points,
    \item every point lies on $t + 1$ lines,
    \item there is at most one point on any two distinct lines, and
    \item if $P$ is a point not on line $\ell$, then there is a unique line incident with $P$ and meeting $\ell$.
\end{itemize}
\tab The \textbf{collinearity graph} of a $GQ(s,t)$ is the graph where the vertices are the set of points in the $GQ(s,t)$ where two points are adjacent if and only if there is some line in the $GQ(s,t)$ that they are commonly incident to.
\begin{figure}[H]
    \begin{subfigure}{.5\textwidth}
      \centering
    \includegraphics[width=0.7\linewidth]{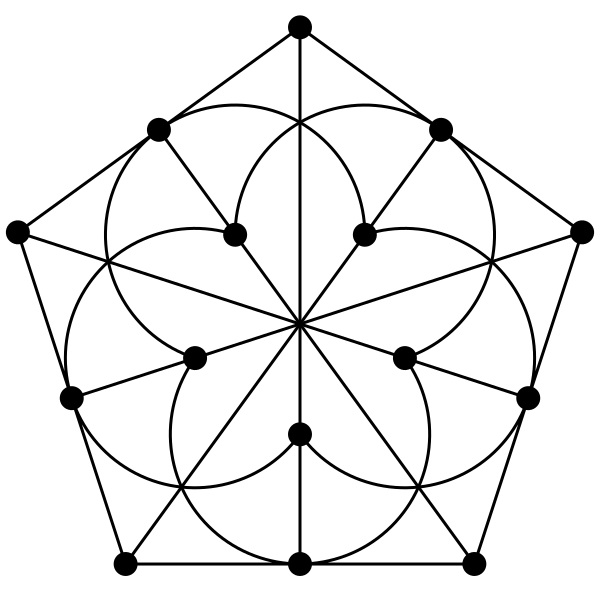}
      \caption{A GQ$(2,2)$ and}
      
    \end{subfigure}
    \begin{subfigure}{.5\textwidth}
      \centering
      \includegraphics[width=0.7\linewidth]{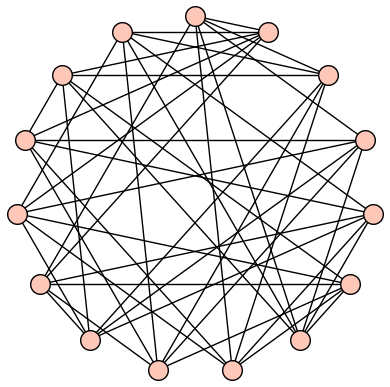}
      \caption{its Colliearity Graph}
      
    \end{subfigure}
    \caption{}
    \end{figure}
    
\begin{thm} \label{thm:gq}
    Suppose $\Gamma$ is the collinearity graph of a $GQ(s,t)$. Then $\Gamma$ is $(s+1)$-clique regular and a regular clique assembly.
\end{thm}
\begin{proof}
     Observe that the $s+1$ points on any line in the $GQ(s,t)$ induce an $(s+1)$-clique in $\Gamma$. It follows from the definition that generalized quadrangles contain no triangles, so the only $(s+1)$-cliques in $\Gamma$ are those induced by the set of points on some line. So if points $v$ and $u$ are adjacent in $\Gamma$ then they are commonly incident to some line $\ell$ in $GQ(s,t)$. So the edge $uv$ is in the $(s+1)$-clique induced by the points on line $\ell$ and points $u$ and $v$ are commonly incident to no other lines, so the edge $uv$ is in a unique $(s+1)$-clique in $\Gamma$.\\
    Since the collinearity graph of a $GQ(s,t)$ is an srg$((s+1)(st+1), s(t+1), s-1, t+1)$, then $\Gamma$ is a regular clique assembly by Theorem \ref{thm:rca}.
\end{proof}
For the following theorem, recall that the \textbf{dual} of a point line incidence structure is the structure formed by swapping the points and the lines of the original structure. If $\beta$ is a point line incidence structure, then point $P$ is on line $\ell$ in $\beta$ if and only if point $\ell$ is on line $P$ in the dual of $\beta$. In the case of generalized quadrangles, it turns out that the dual of a $GQ(s,t)$ is a $GQ(t,s)$ and so has a strongly regular collinearity graph. The following theorem gives an alternate proof that the dual of a $GQ(s,t)$ has a strongly regular collinearity graph.
\begin{thm} \label{thm:gq1}
    If $\Gamma$ is the collinearity graph of a $GQ(s,t)$, then $C_{s+1}(\Gamma)$ is isomorphic to the collinearity graph of the dual of the $GQ(s,t)$ structure that formed $\Gamma$ and is strongly regular with parameters 
    \[\text{srg}(\left(t+1\right)\left(st+1\right), \; t(s+1),\; t-1, \;s+1).\]
\end{thm}
\begin{proof}
    Let $\Gamma$ be the collinearity graph of a $GQ(s,t)$. As shown in Theorem \ref{thm:gq}, there exists a bijection between the lines of the $GQ(s,t)$ and the $(s+1)$-cliques of $\Gamma$ that sends a line in the $GQ(s,t)$ to the set of points in $\Gamma$ that are incident to that line. Since the lines in the $GQ(s,t)$ form the points in the structure's dual, this is a bijection between the vertices of $C_{s+1}(\Gamma)$ and the points on the collinearity graph of the dual of the $GQ(s,t)$. We will show this bijection is an isomorphism. If two points $u$ and $v$ are adjacent in the collinearity graph of the dual of the $GQ(s,t)$, then they are commonly incident to some line $\ell$ in the dual of the $GQ(s,t)$. This means that point $\ell$ is incident to lines $u$ and $v$ in the $GQ(s,t)$ and so the $(s+1)$-cliques they form in $\Gamma$ share a vertex and are adjacent in $C_{s+1}(\Gamma)$. If these two points are not adjacent in the collinearity graph of the dual of the $GQ(s,t)$, then there is no line that they are commonly incident to. So in the $GQ(s,t)$, lines $u$ and $v$ share no common point and so their $(s+1)$-cliques in $\Gamma$ share no vertices and are therefore not adjacent in $C_{s+1}(\Gamma)$. So the collinearity graph of the dual of the $GQ(s,t)$ that formed $\Gamma$ is isomorphic to the $(s+1)$-clique graph of $\Gamma$.\\
    Now we will show that $C_{s+1}(\Gamma)$ is strongly regular. As above, it is well known that $\Gamma$ is an srg$((s+1)(st+1),s(t+1),s-1, t+1)$ so it has the spectrum \[s(t+1)^1,(s-1)^{\frac{st(t+1)(s+1)}{s+t}},(-t-1)^{\frac{s^2(st+1)}{s+t}}.\]
    Since $\Gamma$ is $(s+1)$-clique regular and a regular clique assembly from Theorem \ref{thm:gq}, and the smallest eigenvalue of $\Gamma$ is equal to $-\frac{s(t+1)}{(s+1)-1}=-t-1$, then Theorem \ref{thm:srg} tells us that $C_{s+1}(\Gamma)$ is strongly regular with parameters as above.
\end{proof}
\begin{exmp}
    The collinearity graph $\Gamma$ of a $GQ(3,5)$ is an srg$(64,18,2,6)$. The spectrum of $\Gamma$ is $2^{45}, -6^{18}, 18^1$. Applying Theorem \ref{thm:gq1}, gives us that $C_4(\Gamma$) is the collinearity graph of a $GQ(5,3)$ and is strongly regular with the parameters srg$(96,20,4,4)$ and spectrum
    \[4^{45},-4^{50},20^1\] which is verified using SAGE.
    An added note is that the srg(64,18,2,6) with automorphism group of size 138240 is the collinearity graph of a $GQ(3,5)$ and its $4$-clique graph, an srg(96,20,4,4), has an automorphism group of equal size.\cite{will}
\end{exmp}
\subsection{Locally Linear Graphs}
\tab One of the most interesting families of graphs that these general formulas apply to is the \textbf{locally linear} graphs. A locally linear graph is a graph in which each edge is in a unique triangle. From this definition locally linear graphs are exactly the $3$-clique regular graphs since a $3$-clique is a triangle. Taking the theorems and corollaries for the $\omega$-clique graphs and $\omega$-clique regular graphs and applying them in the case where $\omega=3$ gives theorems that are applicable to locally linear graphs. The 3-clique graph is the graph of triangles where two triangles are adjacent if and only if they share exactly one vertex. The general theorems when applied to the $\omega=3$ case give the following corollaries:
\begin{cor} \label{cor:ll1}
    Suppose $\Gamma$ is locally linear and the eigenvalues of $L(\Gamma)$ are $\mu_1 \leq \cdots \leq \mu_m$. Then for each eigenvalue $\lambda$ of $C_3(\Gamma)$,
    \[\frac{1}{4}\left(3\mu_1 -6\right) \leq \lambda \leq \frac{1}{4}\left(3\mu_m -6\right).\]
\end{cor}
\begin{cor} \label{cor:ll2}
    Let $\Gamma$ be locally linear. Then for each eigenvalue $\lambda$ of $C_3(\Gamma)$,
    \[-3 \leq \lambda \leq 3\left(\frac{\Delta(\Gamma)}{2}-1\right).\]
\end{cor}
\begin{figure}[H]
    \centering
    \includegraphics[width=0.5\linewidth]{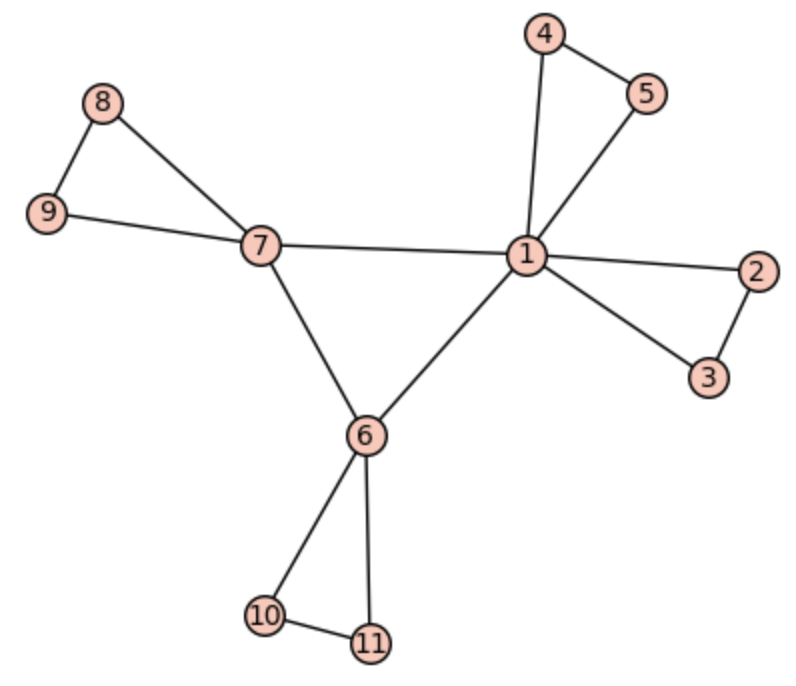}
    \caption{}
    \label{fig:cactus}
\end{figure}

\begin{exmp}
    The figure \ref{fig:cactus} is a locally linear graph $\Gamma$. The $3$-clique graph of $\Gamma$ will have 5 vertices and 5 edges. Corollary \ref{cor:ll1}, tells us that the eigenvalues of $C_3(\Gamma$) will fall in the range $[-3,2.789]$. The largest eigenvalue of $C_3(\Gamma)$ is 2.3429 and the smallest eigenvalue is -1.8136 which both fall within this range.
\end{exmp}
\begin{cor} \label{cor:ll3}
    If $\Gamma$ is $k$-regular and locally linear, then 
    \[p(C_3(\Gamma); \lambda)=(\lambda+3)^{\frac{nk}{6}-n}p\left(\Gamma;\lambda+3-\frac{k}{2}\right).\]
\end{cor}
Note that a strongly regular graph is locally linear precisely if it is an srg($n,k,1,\mu$).
\begin{cor} \label{cor:ll4}
    Suppose $\Gamma$ is a non-boring srg$(n,k,1, \mu)$ with spectrum $k^1,r^f,s^g$ where $r>s$. Then the 3-clique graph of $\Gamma$ is strongly regular if and only if $s = -\frac{k}{2}$ or $k = 6$. If so, $C_3(\Gamma)$ has parameters
    $$\text{srg}\left( \frac{nk}{6},\; \frac{3k-6}{2},\; \frac{k-4}{2},\; \mu+3-\frac{k}{2} \right).$$
\end{cor}
Following from this corollary, we can show that there are only three non-boring locally linear strongly regular graphs with strongly regular $3$-clique graphs.
 \begin{thm}
     The only non-boring strongly regular locally linear graphs that have strongly regular $3$-clique graphs are the unique graphs with the parameters $\text{srg}(9,4,1,2)$, $\text{srg}(15,6,1,3)$, and $\text{srg}(27, 10, 1, 5)$.
 \end{thm}
 \begin{proof}
     These graphs all have strongly regular $3$-clique graphs following from Corollary \ref{cor:ll4}. So let $\Gamma$ be an srg$(n,k,1,\mu)$ such that $C_3(\Gamma)$ is strongly regular and we will show that $\Gamma$ is the unique strongly regular graph on one of the enumerated parameters. From Corollary \ref{cor:ll4} either $k=6$ or $s=-\frac{k}{2}$ so first assume $k=6$. Then from $(n-k-1)\mu=k(k-\lambda-1)$ it follows that $n=\frac{24}{\mu}+7$. So since $\mu$ divides $24$, $\mu \in \{1, 2, 3, 4, 6, 8, 12, 24\}$. The multiplicity of the largest eigenvalue of $\Gamma$ is \[f = \frac{1}{2}\left[\left(\frac{24}{\mu}+6\right)-\frac{12 + (\frac{24}{\mu}+6)(1-\mu)}{\sqrt{(1 - \mu)^2+4(6-\mu)}}\right],\] from which it follows that $f$ is only an integer when $\mu =3$. This shows that $\Gamma$ is the unique srg$(15,6,1,3)$.\\
     Now assume that $s=-\frac{k}{2}$. From this we get
     \[s = \frac{1}{2}\left[(1 - \mu) - \sqrt{(1 - \mu)^2+4(k-\mu)}\right] = -\frac{k}{2},\] which implies that $\mu = \frac{k}{2}$. So from $(n-k-1)\mu=k(k-\lambda-1)$ we get that $\Gamma$ is an srg$(3(k-1), k, 1, \frac{k}{2})$ for some $k$. So the multiplicity of the smallest eigenvalue of $\Gamma$ is
     \[g = \frac{8(k-1)}{k+2}.\] Since $0<g<8$ when $k >1$,  $g \in \{1, \ldots, 7\}$ since the multiplicity must be an integer. Solving for $k$ gives $k\in \{2, 4, 6, 10, 22\}$ since $k$ must also be an integer. If $k=2$ this implies that $\Gamma \cong K_3$, a boring graph, and $k=22$ gives the parameters of $\Gamma$ as srg$(63, 22, 1, 11)$ which violate the absolute bound \cite{brouwer}. The remaining \linebreak $k\in \{4,6,10\}$ give the result.
 \end{proof}
 It should be noted that the unique srg$(27,10,1,5)$ has automorphism group of size $51840$ and its $3$-clique graph is the strongly regular graph on parameters srg$(45,12,3,3)$ with the largest automorphism group, also of size $51840$.\cite{cool}
\begin{figure}[H]
        \centering
        \includegraphics[width=0.5\linewidth]{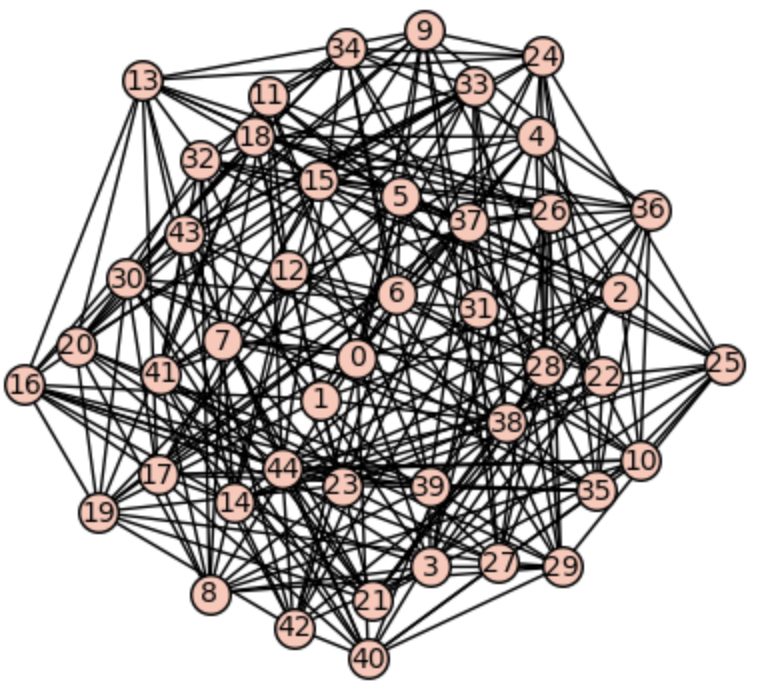}
        \caption{\centering Strongly regular graph with parameters (45,12,3,3) with automorphism group of size 51840}
        \label{fig:srg}
    \end{figure}

\begin{exmp} \label{exmp:1,2s}
    There are only five possible parameter sets of strongly regular graphs with $\lambda=1$ and $\mu= 2$. The strongly regular graphs with parameter sets (9,4,1,2) and (243,22,1,2) are the only two out of the five that are known to exist. Since srg(9,4,1,2) and srg(243,22,1,2) are regular and 3-clique regular, we can apply Corollary \ref{cor:ll3}, giving us that the $3$-clique graph of the unique strongly regular graphs $(9,4,1,2)$ and $(243,22,1,2)$ have spectrum $-3^1, 3^1, 0^4$ and $-3^{648}, 3^{110}, 12^{132}, 30^1$, respectively. Both of these are verified using SAGE. A long standing question in algebraic graph theory pertains to the existence of strongly regular graphs with parameters (99,14,1,2), (6273,112,1,2), and (494019,994,1,2). The existence question of a strongly regular graph with parameters (99,14,1,2) is called Conway's 99-graph problem. If such graphs exist then they would all be regular and locally linear, we can then apply Corollary \ref{cor:ll3} giving us the spectrum of their $3$-clique graphs as follows.
    \begin{center}
        \begin{tabular}{ |c|c| } 
        \hline
        Parameter sets & Spectrum of their 3-clique graph\\
         \hline
         $(9,4,1,2)$ & $-3^1, 3^1, 0^4$\\
         $(99,14,1,2)$ & $-3^{132},7^{54}, 18^1, 0^{44}$ \\ 
         $(243,22,1,2)$ & $-3^{648}, 3^{110}, 12^{132}, 30^1$\\
         $(6273,112,1,2)$ & $-3^{110823}, 42^{2992}, 63^{3280}, 165^1$ \\ 
         $(494019,994,1,2)$ & $-3^{81348462}, 462^{243104} ,525^{250914}, 1488^1$ \\ 
         \hline
        \end{tabular}
    \end{center}
\end{exmp}

\section*{Acknowledgments}
We are grateful to the Robert E. Tickle foundation for funding our research, and to Dr. Joshua Ducey for mentoring us.

\medskip

\printbibliography

\end{document}